\documentclass[11pt, reqno]{amsart}
\usepackage[english]{babel}
\usepackage[colorlinks,citecolor=green,linkcolor=red]{hyperref}
\usepackage{amsthm}
\usepackage{color}
\usepackage{mathrsfs}
\usepackage{amsmath}
\usepackage{mathtools}
\usepackage{amsfonts}
\usepackage{amssymb}
\usepackage{bm}
\usepackage{physics}
\usepackage{enumitem}
\usepackage{tikz-cd}
\usepackage{a4wide}
\usepackage{cleveref}
\usepackage{esint}
\usepackage{nicefrac}
\usepackage{dirtytalk}

\numberwithin{equation}{section}

\newtheorem{thm}{Theorem}[section]

\newtheorem*{thm*}{Theorem}
\newtheorem{lem}[thm]{Lemma}
\newtheorem{prop}[thm]{Proposition}

\newtheorem{defn}[thm]{Definition}

\theoremstyle{remark}
\newtheorem{rem}[thm]{Remark}

\DeclareMathOperator*{\aplimsup}{ap\,\limsup}
\DeclareMathOperator*{\apliminf}{ap\,\liminf}

\newcommand{\Ch}{{\mathrm {Ch}}}
\newcommand{\lip}{{\mathrm {lip}}}
\newcommand{\Lip}{{\mathrm {Lip}}}
\newcommand{\diff}{{\mathrm{d}}}
\newcommand{\DIFF}{{\mathrm{D}}}
\newcommand{\nablatilde}{\bar{\nabla}}
\newcommand{\Prbar}{\bar{\Pr}}
\newcommand{\pibar}{\bar{\pi}}

\newcommand{\qcr}{{\mathrm{QCR}}}
\renewcommand{\tr}{{\mathrm{tr}}}
\newcommand{\trbar}{{{\bar{\tr}}}}
\newcommand{\qcrbar}{{\mathrm{\bar{QCR}}}}
\newcommand{\dive}{{\mathrm{div}}}

\newcommand{\heat}{{\mathrm {h}}}

\newcommand{\supp}{{\mathrm {supp\,}}}

\newcommand{\capa}{{\mathrm {Cap}}}
\newcommand{\per}{{\mathrm {Per}}}

\newcommand{\Wass}{{\mathrm {W}}}
\newcommand{\Prob}{{\mathscr{P}}}
\newcommand{\Ptwo}{{\mathscr{P}_2 }}

\newcommand{\mres}{\mathbin{\vrule height 1.6ex depth 0pt width 0.13ex\vrule height 0.13ex depth 0pt width 1.3ex}}

\newcommand{\dist}{{\mathsf{d}}}
\newcommand{\mass}{{\mathsf{m}}}

\newcommand{\XX}{{\mathsf{X}}}

\newcommand{\WW}{{\mathcal{W}}}

\newcommand{\Tan}{\mathrm {Tan}}

\newcommand{\defeq}{\mathrel{\mathop:}=}

\newcommand{\MM}{\mathcal{M}}

\newcommand{\HH}{\mathcal{H}}
\newcommand{\LL}{\mathcal{L}}
\newcommand{\RR}{\mathbb{R}}
\newcommand{\NN}{\mathbb{N}}
\newcommand{\Cqcvf}{\mathcal{QC}(T\XX)}
\newcommand{\Cqcvfinf}{\mathcal{QC}^\infty(T\XX)}

\newcommand{\Ss}{{\mathrm {S}^2}(\XX)}
\newcommand{\Lp}{{\mathrm {L}}}
\newcommand{\Lploc}{\mathrm {L}_\mathrm{loc}}
\newcommand{\Lpo}{\Lp^0(\mass)}
\newcommand{\Lpc}{\Lp^0(\capa)}

\newcommand{\Lpu}{\Lp^1(\mass)}

\newcommand{\Lpuloc}{\Lp^1_{\mathrm{loc}}(\mass)}
\newcommand{\Lpp}{\Lp^p(\mass)}

\newcommand{\Lpt}{\Lp^2(\mass)}

\newcommand{\Lpi}{\Lp^\infty(\mass)}

\newcommand{\WSCs}{\mathrm {W^{1,2}_C}(T\XX)}

\newcommand{\HSs}{{\mathrm {H^{1,2}(\XX)}}}

\newcommand{\WHCSs}{{\mathrm {H^{1,2}_C}(T\XX)}}
\newcommand{\WSHs}{{\mathrm {W^{1,2}_H}(T\XX)}}
\newcommand{\WHHSs}{{\mathrm {H^{1,2}_H}(T\XX)}}
\newcommand{\WHHSsn}{{\mathrm {H^{1,2}_H}(T\XX)^n}}
\newcommand{\HSsloc}{{\mathrm {H^{1,2}_{loc}(\XX)}}}

\newcommand{\LIP}{{\mathrm {LIP}}}
\newcommand{\BV}{{\mathrm {BV}}}
\newcommand{\BVv}{{\mathrm {BV}}(\XX)}
\newcommand{\LIPbs}{{\mathrm {LIP_{bs}}}}

\newcommand{\Cb}{{C_{\mathrm{b}}}}

\newcommand{\LIPloc}{{\mathrm {LIP_{loc}}}}
\newcommand{\LIPb}{{\mathrm {LIP_{b}}}}

\newcommand{\TestF}{{\mathrm {TestF}}}

\newcommand{\TestV}{{\mathrm {TestV}}}

\newcommand{\TestVbar}{{\mathrm {Test\bar{V}}}}

\newcommand{\cotX}{\Lp^2(T^*\XX)}

\newcommand{\tanX}{\Lp^2(T\XX)}
\newcommand{\tanXp}{\Lp^p(T\XX)}

\newcommand{\tanXzero}{\Lp^0(T\XX)}
\newcommand{\cotanXzero}{\Lp^0(T^*\XX)}
\newcommand{\tanXinf}{\Lp^\infty(T\XX)}

\newcommand{\cotanXp}{\Lp^p(T^*\XX)}

\newcommand{\tanbvX}[1]{\Lp^2_{#1}(T\XX)}
\newcommand{\tanbvXp}[2]{\Lp^{#1}_{#2}(T\XX)}

\newcommand{\tanbvXn}[1]{\Lp^2_{#1}(T^n\XX)}

\newcommand{\tanbvXzero}[1]{\Lp^0_{#1}(T\XX)}

\newcommand{\tanXcap}{\Lp^0_\capa(T\XX)}

\newcommand{\tanXcapinfn}{\Lp^\infty_\capa(T^n\XX)}
\newcommand{\tanXcapinfm}{\Lp^\infty_\capa(T^m\XX)}

\newcommand{\RCD}{{\mathrm {RCD}}}

\newcommand{\PI}{{\mathrm {PI}}}

\newcommand{\fr}{\penalty-20\null\hfill$\blacksquare$}

\def\Xint#1{\mathchoice
	{\XXint\displaystyle\textstyle{#1}}%
	{\XXint\textstyle\scriptstyle{#1}}%
	{\XXint\scriptstyle\scriptscriptstyle{#1}}%
	{\XXint\scriptscriptstyle\scriptscriptstyle{#1}}%
	\!\int}
\def\XXint#1#2#3{{\setbox0=\hbox{$#1{#2#3}{\int}$}
		\vcenter{\hbox{$#2#3$}}\kern-.5\wd0}}

\def\dashint{\Xint-}

\begin{document}
	\title[BV calculus on RCD spaces]{Calculus and fine properties of functions \\of bounded variation on RCD spaces}

\author[C. Brena]{Camillo Brena}
\address{C.~Brena: Scuola Normale Superiore, Piazza dei Cavalieri 7, 56126 Pisa} 
\email{\tt camillo.brena@sns.it}

\author[N. Gigli]{Nicola Gigli}
\address{N.~Gigli: SISSA, Via Bonomea 265, 34136 Trieste} 
\email{\tt ngigli@sissa.it}

	\begin{abstract}
We generalize the classical calculus rules satisfied by functions of bounded variation to the framework of  $\RCD$ spaces. In the infinite dimensional setting we are able to define an analogue of the distributional differential and on finite dimensional spaces we prove fine properties and suitable calculus rules, such as the Vol'pert chain rule for vector valued functions. 	\end{abstract}
\maketitle
\tableofcontents
\section*{Introduction} The structure of functions of bounded variation is well understood on Euclidean spaces (e.g.\ \cite{AFP00,GMSCartCurr}) and, starting from the seminal papers  \cite{MIRANDA2003,amb00,amb01} (see e.g.\ \cite{DiM14a} and references therein), it is now clear that a  reasonable theory can be developed for real valued BV functions defined on arbitrary metric measure spaces. In this more abstract framework, the effectiveness of the theory depends, as reasonable to expect, on the kind of assumptions imposed on the underlying space: typical ones made when studying Sobolev/BV functions are a doubling property of the measure and a (weak, local) Poincar\'e inequality (spaces satisfying these are called PI spaces). 

For what concerns the studies done here, we emphasize a key difference between Sobolev and BV calculus: in the latter case the distributional differential (whatever it is) can be concentrated on negligible sets and as such a clear understanding of it seems unavoidably connected to the study of fine properties of BV functions. As a concrete example of the difficulties this might create, notice that for Sobolev functions over arbitrary spaces  the  Leibniz rule
\[
|\DIFF(fg)|\leq |f||\DIFF g|+|g||\DIFF f|\qquad \text{a.e.\ for every } f\in \Lp^\infty(\mass) \cap {\rm W}^{1,p}({\XX})
\]
always holds, while its analogue for BV functions might fail, even on PI spaces (see  \cite[Example 4.13]{lahti2018sharp}). 

\bigskip

Here we are interested in BV calculus on $\RCD$ spaces (\cite{AmbrosioMondinoSavare13-2}, \cite{Gigli12} after  \cite{Sturm06I,Sturm06II}, \cite{Lott-Villani09} -  see also the survey  \cite{AmbICM}), which are metric measure spaces satisfying, in a synthetic sense, a bound from below for the Ricci curvature. The starting point of our analysis are the recent papers   \cite{ambrosio2018rigidity,bru2019rectifiability,bru2021constancy}, where the theory of sets of finite perimeter has been generalized to the setting of finite dimensional $\RCD$ spaces obtaining in particular  a version of De Giorgi's Theorem and a Gauss--Green integration by parts formula. Since characteristic functions of sets of finite perimeter are `the most irregular BV functions', in some sense, such results are a strong indication that BV calculus  resembling the Euclidean one is possible in this setting.

\bigskip

With this said, we also try to obtain suitable calculus even in possibly infinite dimensional $\RCD$ spaces and in this direction a first result we obtain, valid on arbitrary metric measure spaces, is that
\begin{equation}
\label{eq:TVac}
 \abs{\DIFF F}\ll\capa\qquad\text{for every } F\in \BV(\XX),
\end{equation}
where $\abs{\DIFF F}$ is the total variation of $F$ and $\capa$ the 2-capacity. This was previously known only on PI spaces (see \cite[Section 1.1.3]{bru2019rectifiability} and the references therein). Using \eqref{eq:TVac} in conjunction with the techniques in  \cite{bru2019rectifiability} we obtain the general integration by parts formula
\begin{equation}
\label{eq:intp}
\sum_{i=1}^n \int_\XX F_i \dive \,v_i\dd{\mass}=-\int_\XX v\,\cdot\,\nu_F\dd{\abs{\DIFF F}}
\end{equation}
on any $\RCD(K,\infty)$ space and any $F:{\rm X}\to{\mathbb R}^n$ vector valued BV function. Here $\nu_F$ is a suitable vector field of norm 1 (uniquely) defined $\abs{\DIFF F}$-a.e.\ playing the role of $\frac {{\DIFF F}}{\abs{\DIFF F}}$ and $v$ is an arbitrary vector field `sufficiently smooth' (these concepts can be made rigorous via the notions proposed in \cite{Gigli14}, \cite{debin2019quasicontinuous}). We notice that in the scalar case $n=1$ the proof of \eqref{eq:intp} follows verbatim that of the integration by parts proved in \cite{bru2019rectifiability} taking \eqref{eq:TVac} into account. In the vector valued case some care is needed to define the correct notion of BV function, and in particular of `norm of the differential', for the above to work: see Definition \ref{def:defBVv} and notice that it mimics the relaxation of the (integral of the) Hilbert-Schmidt norm of the differential.

Formula \eqref{eq:intp} and the uniqueness of $\nu_F$ suggest to define the distributional differential of $F$ as the product $\nu_F\abs{\DIFF F}$ (see Definition \ref{mvmeas} for the rigorous meaning of this) and it is then natural to  try to understand how does this object behave and whether is satisfies the well-known properties as in Euclidean spaces. It turns out that on finite dimensional $\RCD$ spaces this is the case and we are able to reproduce  some key classical results. They include the  Vol'pert chain rule for vector valued $\BV$ functions under  post-composition with $C^1$ functions (Proposition \ref{volprop}) and the Leibniz rule
\[
	\DIFF (f g) =\bar{f}\DIFF g+\bar{g}\DIFF f,\qquad\text{for every } f,g\in \Lp^\infty(\mass)\cap \BV({\rm X}),
\]
that in particular implies  the previously unknown bound $\abs{\DIFF (f g)}\le \abs{\bar{f}}\abs{\DIFF g}+ \abs{\bar{g}}\abs{\DIFF f}$ (Proposition \ref{leibnizprop}). Here $\bar f,\bar g$ are the precise representatives of $f,g$, see \eqref{precrep}. For what concerns the Vol'pert chain rule, we notice that the usual proof via blow-up procedure does not seem to work, so we have to argue via a series of intermediate results: we first prove, via an integration by Cavalieri's formula, a chain rule for scalar valued $\BV$ functions, then use it to obtain the Leibniz rule, this in turn gives the chain rule for post-composition with polynomials and finally we argue by approximation.

We emphasize that  the necessity of requiring finite dimensionality of the space is mostly related to the seemingly unavoidable need of representing the total variation of $F$ on the jump part via the codimension one spherical Hausdorff measure and use the results in \cite{ambrosio2018rigidity,bru2019rectifiability,bru2021constancy}.
 
\bigskip

There are several open questions left open by our discussion, including: the validity of the general chain rule as in \cite{AmbDM},  of Alberti's rank one property as in \cite{alberti_1993} and that of better understanding weak objects like the distributional differential of BV functions and their relation with, e.g., distributional differentials of Sobolev functions. These topics are going to be addressed in upcoming papers.

\section*{Acknowledgments}
This work originates from  the Master thesis of the first author done while he was student at SISSA.

We wish to thank Daniele Semola for stimulating conversations we had with him while working at this project.
\section{Preliminaries}
\subsection{Metric measure spaces}
In this note, we consider only complete and separable metric spaces.
A metric measure space is a triplet $(\XX,\dist,\mass)$ where $\XX$ is a set, $\dist$ is a (complete and separable) distance on $\XX$ and $\mass$ is a non negative Borel measure that is finite on balls.  
We adopt the convention that metric measure spaces have full topological support, that is to say that for any $x\in\XX$ and $r>0$, we have $\mass(B_r(x))>0$. Also, to avoid pathological situations, we assume that metric measure spaces are not single points. 

We denote the Borel $\sigma$-algebra of $\XX$ by $\mathcal{B}(\XX)$. For $B$ subset of $\XX$ and $A$ open subset of $\XX$, we write $B\Subset A$ if $B$ is a bounded subset of $A$ with $\dist (B,\XX\setminus A)>0$. Clearly, if the space is proper (i.e.\ bounded sets are relatively compact), $B\Subset A$ if and only if $\bar{B}$ is a compact subset of $A$.

Given $A\subseteq\XX$ open, we denote with $\LIPloc(A)$ the space of Borel functions that are Lipschitz in a neighbourhood of $x$, for any $x\in A$. If the space is locally compact, $\LIPloc(A)$ coincides with the space of functions that are Lipschitz on compact subsets of $A$. We adopt the usual notation for the various Lebesgue spaces. The subscript $\mathrm{bs}$ (e.g.\ $\LIPbs(\XX)$) is used to denote the subspace of functions with bounded support.

A pointed metric measure space is a quadruplet $(\XX,\dist,\mass,x)$ where $(\XX,\dist,\mass)$ is a metric measure space and $x\in\XX$. We consider two pointed metric measure spaces $(\XX',\dist',\mass',x')$ and $(\XX'',\dist'',\mass'',x'')$ to be isomorphic if there exists an isometry
$\Psi:\XX'\rightarrow\XX''$ such that $\Psi(x')=x''$ and $\Psi_*\mass'=\mass''$, where the notation $f_\ast\nu$ denotes the push-forward of the measure $\nu$ through the measurable map $f$.

\bigskip
The Cheeger energy (see \cite{Cheeger00,Shanmugalingam00,AmbrosioGigliSavare11,AmbrosioGigliSavare11-3}) associated to a metric measure space $(\XX,\dist,\mass)$ is the convex and lower semicontinuous functional defined on $\Lpt$ as

\begin{equation}\label{defch}
\Ch(f)\defeq\frac{1}{2}\inf\left\{\liminf_k \int_\XX\lip(f_k)^2\dd{\mass}:\{f_k\}_k\subseteq\LIPb(\XX)\cap\Lpt, f_k\rightarrow f \text{ in }\Lpt\right\}
\end{equation}
where $\lip(f)$ is the so called local Lipschitz constant
$$ \lip(f)(x)\defeq \limsup_{y\rightarrow x}\frac{\abs{f(y)-f(x)}}{\dist(y,x)},$$
which has to be understood to be $0$ if $x$ is an isolated point.
The finiteness domain of the Cheeger energy is denoted by $\HSs$ and is endowed with the complete norm $\Vert f\Vert^2_{\HSs}\defeq\Vert f\Vert_{\Lpt}^2+2 \Ch(f)$. 
It is possible to identify a canonical object $\abs{\nabla f}\in\Lpt$, called minimal relaxed slope, providing the integral representation $$ \Ch(f)=\frac{1}{2}\int_\XX\abs{\nabla f}^2\dd{\mass}\quad\text{for every }f\in\HSs.$$
Any metric measure space on which $\Ch$ is a quadratic form is said to be infinitesimally Hilbertian (\cite{Gigli12}). Under this assumption, (see \cite{Ambrosio_2014,Gigli12}) it is possible to define a symmetric bilinear form
$$ \HSs\times\HSs\ni(g,f)\rightarrow \nabla f\,\cdot\, \nabla g\in\Lpu$$
such that $$ \nabla f\,\cdot\, \nabla f=\abs{\nabla f}^2\quad\mass\text{-a.e.}\text{ for every }f\in\HSs.$$
We denote with $\HSsloc$ the space of functions $f:\XX\rightarrow\RR$ such that for every bounded Borel set $B$, there exists a function $f_B\in\HSs$ such that $f=f_B\ \mass$-a.e.\ on $B$, and we define $\abs{\nabla f}$ exploiting locality. We define $\mathrm{S}^2(\XX)$ as the space of functions $f$ such that for every $n$ $(f\wedge n)\vee -n\in\HSsloc$ and $\abs{\nabla f}\in\Lpt$, where $\abs{\nabla f}$ is (well) defined by $\abs{\nabla ((f\wedge n)\vee -n)}$ $\mass$-a.e.\ on $\{f\in (-n,n)\}$.

On infinitesimally Hilbertian metric measure spaces it is possible to define a linear Laplacian operator $\Delta:D(\Delta)\subseteq\HSs\rightarrow\Lpt$ in the following way:
we let $D(\Delta)$ to be the set of those $f\in\HSs$ such that, for some $h\in\Lpt$, one has 
$$\int_\XX\nabla f\,\cdot\, \nabla g \dd{\mass}= -\int_\XX h g\dd{\mass}\quad\text{for every } g\in\HSs,$$
and, if this is the case, we put $\Delta f=h$, which is uniquely determined by the equation above.

We can define the heat flow $\heat_t$ as the $\Lp^2$ gradient flow of $\Ch$, whose existence and uniqueness follow from the Komura-Brezis theory.  On infinitesimally Hilbertian spaces we can characterize the heat flow by saying that for any $u\in\Lpt$, the curve $[0,\infty)\ni t\mapsto\heat_t u\in\Lpt$ is continuous in $[0,\infty)$, locally absolutely continuous in $(0,\infty)$ and satisfies 
$$\begin{cases}
 \dv{t} \heat_t u=-\Delta \heat_t u\quad\text{for every }t\in(0,\infty),\\
 \heat_0 u=u,
\end{cases}$$ 
where we implicitly state that if $t>0$, $\heat_t u\in D(\Delta)$. 
It is possible to prove that on infinitesimally Hilbertian spaces, the heat flow provides a linear, continuous and self-adjoint contraction semigroup in $\Lpt$, which extends to a linear and continuous contraction semigroup, that we still denote with $\heat_t$, in all spaces $\Lpp$, for $p\in[1,\infty)$. 
We define $\heat_t$ on $\Lpi$ in duality with $\Lpu$,
i.e.\ if $f\in\Lpi$,
$$
\int_\XX g\heat_t f \dd{\mass}=\int_\XX f \heat_t g\dd{\mass}\quad\text{for every }g\in\Lpu,
$$
and, with this extension, $\heat_t$ turns out to be a linear and continuous contraction semigroup in all spaces $\Lpp$ with $p\in [1,\infty]$.

\subsection{RCD spaces}\label{RCDsubsect}
The main setting for our investigation is the one of $\RCD(K,N)$ metric measure spaces (for $K\in\RR$ and $N\in [1,\infty]$), which are infinitesimally Hilbertian spaces (\cite{Gigli12}) satisfying a lower Ricci curvature bound and an upper dimension bound (meaningful if $N<\infty$) in synthetic sense according to \cite{Sturm06I,Sturm06II,Lott-Villani09}. General references on this topic are \cite{AmbrosioGigliMondinoRajala12,AmbrosioGigliSavare11,Ambrosio_2014,AmbrosioGigliSavare12,Gigli17,Gigli14,GP19} and we assume the reader to be familiar with this material.

In the last part of this note the focus be only on finite dimensional $\RCD$ spaces, so that in the sequel when we write $\RCD(K,N)$ we will assume $1\le N<\infty$.
Recall that $\RCD(K,N)$ spaces are locally uniformly doubling, i.e.\ for every $R>0$ there exists $C_D=C_D(R)>0$ such that 
\begin{equation}\notag
\mass(B_{2 r}(x))\le C_D\mass(B_r(x))\quad \text{for every }x\in\XX\text{ and } 0<r<R
\end{equation} 
and support a weak local $(1,1)$-Poincaré inequality, i.e.\ there exists $\lambda\ge 1$ and for every $R>0$, there exists $C_P=C_P(R)>0$ such that, for every $f\in\LIP(\XX)$,
\begin{equation}\label{poincare1}
	\dashint_{B_r(x)}{\abs{f-(f)_{x,r}}\dd{\mass}}\le C_P r \dashint_{B_{\lambda r}(x)} \lip(f)\dd{\mass}\quad \text{for every }x\in\XX\text{ and }\ 0<r<R.
\end{equation}
Hajłasz and Koskela proved in \cite[Theorem 5.1]{hajkosk} that the Poincaré inequality improves to the following form (see also \cite{Cheeger00} for what concerns this formulation): there exists $\lambda\ge 1$ and for every $R>0$, there exist $C'_P=C'_P(R)>0$ and $Q=Q(R)>1$ such that, for every $f\in\LIP(\XX)$,
\begin{equation}\label{poincare2}
\left(\dashint_{B_r(x)}{\abs{f-(f)_{x,r}}}^{\frac{Q}{Q-1}}\dd{\mass}\right)^{\frac{Q-1}{Q}}\le C'_P r \dashint_{B_{2 \lambda r}(x)} \lip(f)\dd{\mass}\quad \text{for every }x\in\XX\text{ and } 0<r<R.
\end{equation}
Recall that locally uniformly doubling spaces are proper. We call locally uniformly doubling spaces supporting a weak local $(1,1)$-Poincaré inequality $\PI$ spaces. 
We can, and will, assume that $R\mapsto C_D(R),R\mapsto C_P(R),R\mapsto C'_P(R)$ and $R\mapsto Q(R)$ are non decreasing functions. 

Following \cite{Gigli14,Savare13} (with the additional request of a $\Lp^\infty$ bound on the Laplacian), we define the vector space of test functions on an $\RCD(K,\infty)$ space as
	\begin{equation}\notag
		\TestF(\XX)\defeq\{f\in\LIP(\RR)\cap\Lpi\cap D(\Delta): 	\Delta f\in \HSs\cap\Lpi\},
	\end{equation}
and the vector space of test vector fields as
	\begin{equation}\label{testVdef}
		\TestV(\XX)\defeq\left\{ \sum_{i=1}^n f_i\nabla g_i : f_i\in\Ss\cap\Lpi,g_i\in\TestF(\XX)\right\}.
	\end{equation}
To be precise, the original definition of $\TestV(\XX)$ given by the second author was slightly different. However, when using test vector fields to define regular subsets of vector fields such as $\WHHSs$ and $\WHCSs$, the two definitions produce the same subspaces, as one may readily check inspecting the proofs of Lemma \ref{calculushodge} and Lemma \ref{calculuscova}  below.

It is possible to see that $\TestF(\XX)\subseteq\HSs$ is dense. Also, if $f\in\HSs\cap\Lpi$, we can find a sequence $\{f_n\}_n\subseteq\TestF(\XX)$ with $f_n\rightarrow f$ in $\HSs$ and $\Vert f_n\Vert_{\Lpi}\le \Vert f\Vert_{\Lpi}$.
Using \cite[Theorem 6.1.11]{GP19} (extracted from \cite{Savare13}), one proves that $	\TestF(\XX)$ is an algebra. Clearly, if $f\in\Ss\cap\Lpi\supseteq\TestF(\XX)$ and $v\in\TestV(\XX)$, then $f v \in\TestV(\XX)$. 
\bigskip

On $\RCD(K,\infty)$ spaces, (\cite{AmbrosioGigliMondinoRajala12,Ambrosio_2014,AmbrosioGigliSavare11,Gigli10}), we can define the heat flow on Borel probability measures with finite second moment (we still denote it, with a slight abuse, by $\heat_t$) as the $\mathrm{EVI}_K$ gradient flow of the entropy, and it turns out that the existence of this gradient flow for any initial datum can be used to characterize $\RCD(K,\infty)$ spaces among length spaces with a growth condition on the reference measure. It is possible to show that, given $\mu\in\Ptwo(\XX)$ and $t>0$, $\heat_t \mu$ is the unique measure in $\Ptwo(\XX)$ that satisfies
	\begin{equation}\label{heatmeassposta}
		\int_\XX g\dd{\heat_t \mu}=\int_\XX \heat_t g\dd{\mu}\quad \text{ for every }g\in\LIPbs(\XX),
	\end{equation}
where we took the Lipschitz representative for $\heat_t g$ in the integral above thanks to the $\Lp^\infty$-$\LIP$ regularization property of the heat flow on $\RCD(K,\infty)$ spaces. The heat flow of measures is $K$-contractive with respect to the Wasserstein $\Wass_2$ distance and, for $t > 0$, maps probability measures into probability measures which are absolutely continuous with respect to $\mass$ (the latter assertion is an immediate consequence of fact that $t\mapsto\heat_t\mu$ is the gradient flow of the entropy). Then, for any $t > 0$, it is possible to define the heat kernel $p_t : \XX \times \XX \rightarrow [0,\infty)$ by
$$p_t(x,\,\cdot\,)\defeq\dv{\heat_t\delta_x}{\mass}.$$
As $\RCD(K,N)$ spaces are $\PI$ spaces, the theory developed in \cite{Sturm96II,Sturm96III} implies the existence of a locally Hölder continuous representative for the heat kernel $$(0,\infty)\times\XX\times\XX\ni(t,x,y)\mapsto p_t(x,y)\in\RR.$$ In \cite{jiang2014heat} it has been proved that, for any $\epsilon>0$, there exist positive constants $C_1=C_1(\epsilon,K,N)>0$ and $C_2=C_2(\epsilon,K,N)>0$ such that for every $t>0$, $x,y\in\XX$, the following estimate holds
\begin{equation}\label{heatkerneleq}
		\frac{1}{C_1\mass(B_{\sqrt{t}}(y))}\exp{-\frac{\dist(x,y)^2}{(4-\epsilon)t}-C_2 t}\le p_t(x,y)\le \frac{C_1}{\mass(B_{\sqrt{t}}(y))}\exp{-\frac{\dist(x,y)^2}{(4+\epsilon)t}+C_2 t}.
\end{equation}
Then, if $\mu$ is a finite (non negative) Borel measure on $\XX$, we can define 
\begin{equation}\notag
	\heat_t \mu\defeq\left( \int_\XX p_t(\,\cdot\,,y)\dd\mu{(y)}\right)\mass
\end{equation}
and Fubini's Theorem implies that \eqref{heatmeassposta} still holds, so that this definition is coherent with the previous one. Notice that still $\heat_t\mu\ll\mass$ for every finite (non negative) Borel measure $\mu$.

On an $\RCD(K,\infty)$ space $(\XX,\dist,\mass)$, following \cite{Gigli14}, one can consider $\heat_{\mathrm{H},t}$, the gradient flow relative to the augmented Hodge energy functional in $\tanX$. This means that for every $v\in\tanX$ the curve $t\mapsto\heat_{\mathrm{H},t}v\in\tanX$ is the unique curve that is continuous in $[0,\infty)$, locally absolutely continuous in $(0,\infty)$ and satisfies  $$ \begin{cases}
 	\dv{t}\heat_{\mathrm{H},t}v=-\Delta_\mathrm{H}\heat_{\mathrm{H},t}(v)\quad\text{for every }t\in(0,\infty),\\
 	\heat_{\mathrm{H},0}v=v,
 \end{cases}$$
where we implicitly state that if $t>0$, $\heat_{\mathrm{H},t}v\in D(\Delta_\mathrm{H})\subseteq\WHCSs$.
 
In \cite{Gigli14} and \cite[Section 1.4]{bru2019rectifiability} are proved several properties of the heat flow $\heat_{\mathrm{H},t}$, we recall here some of them. 
The first is the pointwise estimate for $v\in\tanX$
$$
\abs{\heat_{\mathrm{H},t}v}^2\le e^{-2 K t} \heat_t (\abs{v}^2)\quad\mass\text{-a.e.\ for every }t\ge 0.
$$
Then we recall that $\heat_{\mathrm{H},t}$ is self-adjoint, meaning that for every $v,w\in\tanX$,
$$ \int_\XX \heat_{\mathrm{H},t} v\,\cdot\, w\dd{\mass}= \int_\XX v\,\cdot\,  \heat_{\mathrm{H},t}w\dd{\mass}\quad\text{for every }t\ge 0.$$
Also, we recall the commutation, for $v\in D(\dive)$,
$$ \dive(\heat_{\mathrm{H},t} v)=\heat_t( \dive\, v)\quad\mass\text{-a.e.\ for every }t\ge 0,$$
where we recall $\heat_{\mathrm{H},t} v\in D(\Delta_{\mathrm{H}})\subseteq D(\dive)$.
Finally we state that if $f\in\HSs$, then
$$ \heat_{\mathrm{H},t}(\nabla f)=\nabla\heat_t f\quad\text{for every }t\ge 0.$$
\bigskip

We recall now the definition of tangent cone to an $\RCD(K,N)$ space, using the notion of pointed measured Gromov-Hausdorff convergence, defined first in \cite{Gromov07} (see also \cite{Sturm06I, GMS15}). 
First, given a pointed metric measure space $(\XX,\dist,\mass,x)$ and $r\in(0,1)$ we define the rescaled space $(\XX,r^{-1}\dist,\mass^x_r,x)$ where $$\mass^x_r\defeq\left( \int_{B_r(x)}\big(1-r^{-1}\dist(x,z)\big)\dd{\mass(z)}\right)^{-1}\mass.$$
The transformation from $\mass$ to $\mass^x_r$ is performed in order to have the space normalized, i.e.\ 
$$ \int_{B_1^{r^{-1}\dist}(x)}\big(1-r^{-1}\dist(x,z)\big)\dd{\mass_r^x(z)}=1.$$
As a notation, we set $$\tilde{\LL}^k\defeq({\LL}^k)^0_1={\LL}^k\frac{k+1}{\omega_k},$$ where $\LL^k$ denotes the $k$ dimensional Lebesgue measure and
\begin{equation}\label{omegakdef}
	\omega_k\defeq \LL^k(B_1^{\RR^k}(0)).
\end{equation}
\begin{defn}
Let $(\XX,\dist,\mass)$ be an $\RCD(K,N)$ space and $x\in\XX$. We say that a pointed metric measure space $(\XX',\dist',\mass',x')$ is tangent to $(\XX,\dist,\mass)$ at $x$ if there exists a sequence of radii $\{r_j\}_j$, $r_j\searrow 0$ such that $(\XX,r_j^{-1}\dist,\mass_{r_j}^x,x)\rightarrow(\XX',\dist',\mass',x')$ in the pointed measured Gromov-Hausdorff topology.
We denote the collection of all tangent spaces to $(\XX,\dist,\mass)$ at $x$ as $\Tan_{x}(\XX,\dist,\mass)$.
\end{defn}
If $(\XX,\dist,\mass)$ is an $\RCD(K,N)$ space, Gromov compactness theorem shows that for every $x\in\XX$, $\Tan_{x}(\XX,\dist,\mass)$ is non empty. Moreover, by the stability and rescaling property of the $\RCD(K,N)$ condition, we see that elements of $\Tan_{x}(\XX,\dist,\mass)$ are $\RCD(0,N)$ spaces.

The known results of structure theory for $\RCD(K,N)$ spaces can be summed up in the following theorem, which, in particular, states that $\RCD(K,N)$ spaces are rectifiable as metric measure spaces (see \cite{BruPasSem20,bru2018constancy,Gigli_2015,KelMon16,Mondino-Naber14,GP16-2}):
\begin{thm}
	Let $(\XX,\dist,\mass)$ be an $\RCD(K,N)$ space. Then there exists a unique $n\in\NN$, called the essential dimension of $\XX$, with $1\le n\le N$, such that: 
	\begin{enumerate}[label=\roman*)]
		\item for $\mass$-a.e.\ $x\in \XX$,
		\begin{equation}\label{denfrn}
			\Tan_x (\XX,\dist,\mass)=\left\{(\RR^{n},\dist_e,\tilde{\LL}^{n},0)\right\},
		\end{equation}
		and we call the collection of points $x\in\XX$ satisfying \eqref{denfrn} above $\mathcal{R}_n$.
		\item $(\XX,\dist,\mass)$ is countably $n$-rectifiable. More precisely, given any $\epsilon>0$, we can cover $(\XX,\dist,\mass)$ up to negligible subset by a countable union of subsets that are $(1+\epsilon)$-bilipschitz equivalent to measurable subsets of $\RR^n$.
		\item There exists a non negative density $\theta\in\Lploc^1(\XX,\HH^n\mres\mathcal{R}_n)$ such that
		\begin{equation}\notag
			\mass=\theta\HH^n\mres\mathcal{R}_n.
		\end{equation}
	\end{enumerate}
\end{thm}
If $(\XX,\dist,\mass)$ is an $\RCD(K,N)$ space of essential dimension $n$, it holds that for every $x\in\XX$, 
\begin{equation}\notag
(\RR^k,\dist_e,\tilde{\LL}^k,0)\notin\Tan_{x}(\XX,\dist,\mass)\quad\text{ if $k>n$}.
\end{equation}

\subsection{Normed modules}
We assume that the reader is familiar with the notion of normed module, introduced in \cite{Gigli14}, inspired by the theory developed in \cite{Weaver01}. Also, we assume familiarity with the definition of capacitary modules, quasi-continuous functions and vector fields and related material in \cite{debin2019quasicontinuous}. A summary of the material we use can be found in \cite[Section 1.3]{bru2019rectifiability}. For the reader's convenience, we write the results that we will use most frequently.

It is possible to prove that there exists a unique couple $(\cotX,\diff)$ where $\cotX$ is a $\Lp^2$-normed $\Lp^\infty$-module and $\diff:\HSs\rightarrow\cotX$ is linear and such that 
\begin{enumerate}[label=\roman*)]
\item $\abs{\diff f}=\abs{\nabla f}\ \mass$-a.e.\ for every $f\in\HSs$,
\item $\cotX$ is generated (in the sense of modules) by $\left\{\diff f:f\in\HSs\right\}$.
\end{enumerate}
We define the tangent module $\tanX$ as the dual (in the sense of modules) of $\cotX$.
We define $\cotanXzero$ as the $\Lp^0$-completion of the cotangent module $\cotX$ and also (this definition coincides with the previous one if $p=2$)
\begin{equation}\notag
 \cotanXp\defeq\left\{v\in\cotanXzero:\abs{v}\in\Lpp\right\}\quad\text{for $p\in[1,\infty]$}.
\end{equation}
Similarly, we define $\tanXzero$ as the $\Lp^0$-completion of $\tanX$ and
$$  \tanXp\defeq\left\{v\in\tanXzero:\abs{v}\in\Lpp\right\}\quad\text{for $p\in[1,\infty]$}.$$
We also remark that our definition of the tangent and cotangent modules is, in general, different from the one given in \cite{buffa2020bv} for $p\ne 2$. If the space is infinitesimally Hilbertian, it turns out that $\cotX$ is a Hilbert module so that we can, and will, identify $\cotX$ with its dual $\tanX$, via a map that sends $\diff f$ to $\nabla f$ (the latter vector field being given by Riesz Theorem).

\begin{defn}\label{divedefn}
Let $p\in\{2,\infty\}$. For $v\in\mathrm{L}^p(T\XX)$ we say that $v\in D(\dive^p)$ if there exists a function $g \in\Lpp$ such that 
\begin{equation}\label{divedefneq}
\int_\XX \diff f(v)\dd{\mass}=-\int_\XX  f g \dd{\mass}\quad\text{for every }f\in\HSs\text{ with bounded support},
\end{equation}
and such $g$, which is uniquely determined, is denoted by $\dive\, v$.
\end{defn}
Notice that if $v\in D(\dive^2)\cap D(\dive^\infty)$, then the the two objects $\dive\,v$ as above coincide, in particular, $\dive\,v\in\Lpt\cap\Lpi$.
From \eqref{divedefneq} it follows that $\supp(\dive\,v)\subseteq\supp v$ and also notice that, if the space is infinitesimally Hilbertian and $p=2$ (then $\LIPbs(\XX)\subseteq\HSs$ is dense, as a consequence of the result in \cite[Section 8.3]{AmbrosioGigliSavare11-3} or \cite{ACM14}), \eqref{divedefneq} reads
\begin{equation}\notag
	\int_\XX \nabla f\,\cdot\,v\dd{\mass}=\int_\XX  f g \dd{\mass}\quad\text{for every }f\in\LIPbs(\XX).
\end{equation}
Also, the classical calculus rule holds: if $v\in D(\dive^\infty)$ and $f\in\LIPb(\XX)$, then $f v\in D(\dive^\infty)$ and
\begin{equation}\label{calcdiveeq}
\dive(f v)=\diff f(v)+f\dive\, v.
\end{equation}  This follows from \eqref{divedefneq} and the fact that if $g\in\HSs$ has bounded support and $f\in\LIPb(\XX)$, then $f g\in\HSs$ has bounded support and satisfies $\diff(f g)=f\diff g+g\diff f$.
In the case $p=2$, again from the algebra properties of bounded Sobolev functions together with an easy approximation argument, we have that if $v\in D(\dive^2)\cap\tanXinf$ and $f\in\Ss\cap\Lpi$, then $f v\in D(\dive^2)$ and the calculus rule above holds. In the case $p=2$, we often omit to write the superscript $2$ for what concerns the divergence. 
For future reference we recall that, in the particular case of an infinitesimally Hilbertian space and $p=2$, we can write the calculus rule above as follows.
\begin{lem}\label{calculusdive}
	Let $(\XX,\dist,\mass)$ be an infinitesimally Hilbertian space, $v\in D(\dive)\cap\tanXinf$ and $f\in \mathrm{S}^2(\XX)\cap\Lpi$. Then $f v\in D(\dive)$ and 
	$$\dive (f v)=\nabla f\,\cdot\,v+f\dive\,v.$$
\end{lem}
%
%


\subsection{Functions of bounded variation}
We assume that the reader is familiar with the theory of functions of bounded variation and sets of (locally) finite perimeter in metric measure spaces developed in \cite{amb00,amb01,MIRANDA2003} and in the more recent \cite{ambrosio2018rigidity,bru2019rectifiability} for what concerns the $\RCD(K,N)$ setting. 
We recall now the main notions. 

Fix a metric measure space $(\XX,\dist,\mass)$. 
Given $f\in\Lpuloc$, we define, for any $A\subseteq\XX$ open,
\begin{equation}\label{deftv}
	\abs{\DIFF f}(A)\defeq\inf\left\{\liminf_k \int_\XX\lip(f_k)\dd{\mass} :\{f_k\}_k\subseteq\LIPloc(A), f_k\rightarrow f \text{ in }\Lploc^1(A,\mass)\right\}
\end{equation}
($f_k\rightarrow f \text{ in }\Lploc^1(A,\mass)$ if for every $x\in A$ there exists a neighbourhood $U=U_x$ such that $f_k\rightarrow f$ in $\Lp^1(U,\mass)$).
We say that $f$ is a function of bounded variation, i.e.\ $f\in\BVv$, if $f\in\Lpu$ and $\abs{\DIFF f}(\XX)<\infty$. In this case it is easy to show that in \eqref{deftv} $\Lp^1$ convergence can be equivalently taken instead of $\Lploc^1$ convergence. 
We also remark that if $f\in\BV(\XX)$ and $\{f_k\}_k\subseteq\LIPloc(\XX)\cap\Lpu$ is an optimal sequence for the computation of $\abs{\DIFF f}(\XX)$ as in \eqref{deftv}, i.e.\ $f_k\rightarrow f$ in $\Lpu$ and $(\lip(f_k)\mass)(\XX)\rightarrow\abs{\DIFF f}(\XX)$ (by the results in  \cite{DiM14a}, this happens for a sequence $\{f_k\}_k\subseteq\LIPbs(\XX)$), it holds that $\lip(f_k)\mass\rightharpoonup\abs{\DIFF f}$ in duality with $\Cb(\XX)$.

 If $f=\chi_E$, we say that $E$ is a set of locally finite perimeter if $\abs{\DIFF\chi_E}(A)<\infty$ for every $A$ bounded open subset of $\XX$ and we say that $E$ is a set of finite perimeter if $\abs{\DIFF\chi_E}(\XX)<\infty$.
 
If $f\in\BVv$ or $f=\chi_E$, with $E$ set of locally finite perimeter, $\abs{\DIFF f}(\,\cdot\,)$ turns out to be the restriction to open sets of a Borel measure (finite or locally finite) that we denote with the same symbol and we call total variation. If $f=\chi_E$, we denote $\abs{\DIFF f}(\,\cdot\,)$ also with $\per(E,\,\cdot\,)$. 

Notice that, by its very definition, the total variation is lower semicontinuous with respect to $\Lploc^1$ convergence, is subadditive and $\abs{\DIFF (\phi \circ f)}\le L \abs{\DIFF f}$ whenever $f\in\BVv$ and $\phi$ is $L$-Lipschitz. Finally, \eqref{poincare2} and, in particular, \eqref{poincare1} extend immediately to the case $f\in\BVv$, with $\frac{1}{\mass(B_{c r}(x))}\abs{\DIFF f}(B_{c r}(x))$ in place of $\dashint_{B_{c r}(x)} \lip(f)\dd{\mass}$.

Several classical results have been generalized to the abstract framework of metric measure spaces.
Among them, the Fleming-Rishel coarea formula, which states that given $f\in\BVv$, the set $\{f>r\}$ has finite perimeter for $\LL^1$-a.e.\ $r\in\RR$ and
\begin{equation}
	\label{coareaeq}
	\int_\XX h\dd{\abs{\DIFF f}}=\int_\RR\dd{r} \int_\XX h\dd{\per(\{f>r\},\,\cdot\,)}\quad\text {for any Borel function $h:\XX\rightarrow[0,\infty].$}
\end{equation}
In particular,
\begin{equation}
	\label{coareaeqdiff}
	{\abs{\DIFF f}}(A)=\int_\RR\dd{r}  {\per(\{f>r\},A)}\quad\text {for any $A\subseteq \XX$ Borel}.
\end{equation}
A standard consequence of the coarea formula is that given $x\in\XX$, then for $\LL^1$-a.e\ $r\in(0,\infty)$ the ball $B_r(x)$ has finite perimeter. In the framework of $\RCD(K,N)$ spaces this conclusion holds for every $r\in(0,\infty)$ and the Bishop-Gromov inequality provides sharp upper bounds for perimeters of balls.
We also recall that sets of finite perimeters are an algebra, more precisely, if $E$ and $F$ are sets of finite perimeter, then  
\begin{equation}\notag
	\per(E,\,\cdot\,)=\per(\XX\setminus E,\,\cdot\,)\quad\text{and}\quad\per(E\cap F,\,\cdot\,)+\per(E\cup F,\,\cdot\,)\le\per(E,\,\cdot\,)+\per(F,\,\cdot\,).
\end{equation}

\bigskip

Given a measurable set $E$, we define its essential boundary as
\begin{equation}\label{defnessboundary}
	\partial^* E\defeq\left\{x\in\XX: \limsup_{r\searrow 0} \frac{\mass(B_r(x)\cap E)}{\mass (B_r(x))}>0\text{ and } \limsup_{r\searrow 0}\frac{\mass(B_r(x)\setminus E)}{\mass (B_r(x))}>0 \right\},
\end{equation}
and given a measurable function $f:\XX\rightarrow\RR$, we define the approximate lower and upper limits as 
\begin{alignat*}{3}
	f^{\wedge}(x)&\defeq \apliminf_{y\rightarrow x} f(y)&&\defeq\sup&&\left\{t\in\bar{\RR}: \lim_{r\searrow 0} \frac{\mass(B_r(x)\cap\{f<t\})}{\mass(B_r(x))}=0\right\}, \\
	f^{\vee}(x)&\defeq \aplimsup_{y\rightarrow x} f(y)&&\defeq\inf&&\left\{t\in\bar{\RR}:\lim_{r\searrow 0} \frac{\mass(B_r(x)\cap\{f>t\})}{\mass(B_r(x))}=0\right\}.
\end{alignat*}
Notice that if $E$ is a measurable subset of $\XX$, then
\begin{equation}\label{partialstar}
	\partial^*E=\{x: (\chi_E^\wedge(x), \chi_E^\vee(x))=(0,1)\}.
\end{equation}
We set 
\begin{equation}\label{jumpf}
	S_f\defeq\left\{x\,:\,f^{\wedge}(x)<f^{\vee}(x) \right\}.
\end{equation}
If $x\in\XX\setminus S_f$, then $f^{\wedge}(x)=f^{\vee}(x)$ and we denote their common value by $\bar{f}(x)$. If $x\in S_f$ we define \begin{equation}\label{precrep}
	\bar{f}(x)\defeq\dfrac{f^{\wedge}(x)+f^{\vee}(x) }{2},
\end{equation} adopting the convention $+\infty-\infty=0$. We call $\bar{f} $ the precise representative of $f$.

It is possible to prove (see \cite[Lemma 3.2]{kinkorshatuo}) that if $(\XX,\dist,\mass)$ is a $\PI$ space and $f\in\BVv$,
\begin{equation}\label{finitenessprec}
	-\infty<f^\vee(x)\le f^\wedge(x)<\infty\quad\text{ for }\HH^h\text{-a.e.\ } x\in\XX.
\end{equation}
It is well known that $S_f$ can be written as a countable union of reduced boundaries of sets of finite perimeter. It is worth noticing that then Theorem \ref{rectredbound} below implies the rectifiability of $S_f$ in the $\RCD(K,N)$ setting.

\bigskip

We recall that \cite[Lemma 5.2]{amb01} and \cite[Theorem 1.12]{bru2019rectifiability} show that, in the framework of $\PI$ spaces (in particular, in the framework of $\RCD(K,N)$ spaces),
\begin{equation}\notag
	\per(E,\,\cdot\,)\ll\HH^h\ll\capa,
\end{equation}
where $\HH^h$ is the codimension one spherical Hausdorff measure, defined as
$$\HH^h(A)\defeq\lim_{\delta\searrow 0}\HH^h_\delta(A) $$
where (allowing some radius $r_i$ to be $0$)
$$ \HH^h_\delta(A)\defeq\inf\left\{\sum_{i\in\NN} \frac{\mass(B_{r_i}(x_i))}{r_i}:A\subseteq \bigcup_{i\in\NN} {B_{r_i}(x_i)},\ r_i\le \delta\right\}.$$
Recalling the coarea formula \eqref{coareaeqdiff},
\begin{equation}\label{diffllxxxeq}
	\abs{\DIFF f}\ll\HH^h\ll\capa
\end{equation}
whenever $f\in\BVv$. However, we will prove in Theorem \ref{diffllcapthm} that 
\begin{equation}\label{diffllcapeq}
	\abs{\DIFF f}\ll\capa
\end{equation}
holds for any m.m.s.\ $(\XX,\dist,\mass)$ and $f\in\BVv$.

\bigskip

In \cite{ambrosio2018rigidity}, the authors defined a notion of fine tangent bundle tailored for subsets of locally finite perimeter. More precisely (the notions of convergence can be found e.g.\ in \cite{GMS15,AH16}),
\begin{defn}
	\label{defntanfp}
	Let $(\XX,\dist,\mass)$ be an $\RCD(K,N)$ space, $x\in\XX$ and $E$ a measurable subset of $\XX$. We say that the quintuple $(\XX',\dist',\mass',x',E')$ is tangent to $(\XX,\dist,\mass,E)$ at $x$ if there exists a sequence of radii $\{r_j\}_j$, $r_j\searrow0$, such that
	\begin{enumerate}[label=\roman*)]
		\item $(\XX,{r_j}^{-1}\dist,\mass^x_{r_j},x)\rightarrow (\XX',\dist',\mass',x')$ in the pointed measured Gromov-Hausdorff topology,
		\item
		$E'$ is measurable subset of $\XX'$ such that $\{E_j\defeq E\}_j$ converges in $\Lploc^1$ to $E'$ along the sequence of rescaled spaces as in $i)$.
	\end{enumerate} 
	We denote the collection of all tangent spaces to $(\XX,\dist,\mass,E)$ at $x$ as $\Tan_x(\XX,\dist,\mass,E)$.
\end{defn}
We consider two elements $(\XX',\dist',\mass',x',E'),\ (\XX'',\dist'',\mass'',x'',E'')\in\Tan_x(\XX,\dist,\mass,E)$ to be isomorphic if there exists an isometry $\Psi:\XX'\rightarrow\XX''$ such that $\Psi(x')=x''$, $\Psi_*\mass'=\mass''$ (i.e.\ $(\XX',\dist',\mass',x')$ and $(\XX'',\dist'',\mass'',x'')$ are isomorphic as pointed metric measure spaces) and $\mass''(\Psi(E')\Delta E'')=0$.
\begin{thm}[{\cite[Theorem 3.2]{bru2019rectifiability} and \cite[Theorem 3.1]{bru2021constancy}}]
	\label{uniquenesstaper}
	Let $(\XX,\dist,\mass)$ be an $\RCD(K,N)$ space of essential dimension $n$ and let $E\subseteq\XX$ be a subset of locally finite perimeter. Then, for $\abs{\DIFF \chi_E}$-a.e.\ $x\in\XX$ it holds
	\begin{equation}\label{tangood}
		\Tan_{x}(\XX,\dist,\mass,E)\defeq\left\{(\RR^n,\dist_e,\tilde{\LL}^n,0,\{x_n>0\})\right\}.
	\end{equation}
\end{thm}
We call the collection of points $x\in\XX$ satisfying \eqref{tangood} above $\mathcal{F}_n E$ and we see then that $\abs{\DIFF \chi_E}$ is concentrated on $\mathcal{F}_n E$, if $E$ is a subset of locally finite perimeter.

To prove \cite[Theorem 4.2]{ambrosio2018rigidity} (which is an intermediate step in the proof of the theorem above), the authors used the following fact, which follows from the compactness result in \cite[Corollary 3.4]{ambrosio2018rigidity} (to this end, one uses \cite[(5.2)]{amb01}).
\begin{lem}
	\label{compactnessfp}
	Let $(\XX,\dist,\mass)$ be an $\RCD(K,N)$ space and $E$ a subset of locally finite perimeter of $\XX$. For $\abs{\DIFF\chi_E}$-a.e.\ $x\in\XX$ the following conclusion holds. If $\{r_j\}_j$, with $r_j\searrow 0$, is a sequence of radii, then there exists a subsequence $\{r_{j_m}\}_m$ and a quintuple $(\XX',\dist',\mass',x',E')$ that is tangent to $(\XX,\dist,\mass,E)$ at $x$ according to the subsequence $\{r_{j_m}\}_m$.
\end{lem}

\begin{thm}[{\cite[Theorem 4.1]{bru2019rectifiability}}]
	\label{rectredbound}
	Let $(\XX,\dist,\mass)$ be an $\RCD(K,N)$ space of essential dimension $n$ and let $E$ be a subset of locally finite perimeter of $\XX$. Then $\mathcal{F}_n E$ is $(n-1)$-rectifiable. More precisely, for any $\epsilon>0$, we can cover
	$\mathcal{F}_n E$, up to a $\abs{\DIFF \chi_E}$-negligible subset, by a countable union of subsets that are $(1+\epsilon)$-bilipschitz equivalent to measurable subsets of $\RR^{n-1}$.
\end{thm}

Given a set of locally finite perimeter $E$ in an $\RCD(K,N)$ space $(\XX,\dist,\mass)$,
\cite[Corollary 3.15]{bru2019rectifiability} (see also \cite[Corollary 3.2]{bru2021constancy}) proves that 
\begin{equation}\label{reprformula1}
	\abs{\DIFF\chi_E}=\frac{\omega_{n-1}}{\omega_n}\HH^h\mres\mathcal{F}_n E.
\end{equation}

In \cite[Theorem 5.3]{amb01} (see also \cite[Theorem 4.6]{ambmirpal04}), the following representation formula was given:
\begin{equation}\label{reprformula2}
	\abs{\DIFF\chi_E}=\theta_E \HH^h\mres\partial^*E.
\end{equation}
for some Borel function $\theta_E:\XX\rightarrow[\alpha,\beta]$ where $0<\alpha<\beta<\infty$. Comparing the two representations above, it follows that 
\begin{equation}\label{reprformula3}
\abs{\DIFF\chi_E}=\frac{\omega_{n-1}}{\omega_n} \HH^h\mres\partial^*E.
\end{equation}

Using the compactness result in Lemma \ref{compactnessfp} together with the rigidity property in Theorem \ref{uniquenesstaper}, it is easy to prove that
\begin{equation}\notag
	\text{$\abs{\DIFF\chi_E}$-a.e. $x$ is a point of density $1/2$ of $E$.}
\end{equation}  
Also, taking into account \eqref{reprformula3}, we obtain that
\begin{equation}\label{unmezzoqoeq}
	\text{$\HH^h$-a.e.\ $x\in\partial^* E$ is a point of density $1/2$ of $E$.}
\end{equation}

\section{The theory for general metric measure spaces}
\subsection{Basic knowledge}
The following representation formula for the total variation is based on a result proved in \cite{DiMarino14} and then modified in \cite{buffa2020bv} (see \cite[Remark 3.18]{buffa2020bv}). In the particular setting of $\RCD(K,\infty)$ spaces, it is possible to use an approximation argument to provide a direct proof (cf.\ Proposition \ref{reprvett1}).
\begin{prop}[Representation formula]\label{reprfordiffregularpre2}
	Let $(\XX,\dist,\mass)$ be a metric measure space and $f\in\BVv$. Then, for every $A$ open subset of $\XX$, it holds that
	\begin{equation}\label{intagainst}
		\abs{\DIFF f}(A)=\sup\left\{\int_A f \dive\, v\dd{\mass}\right\},
	\end{equation}
	where the supremum is taken among all $v\in\mathcal{W}_A$, where
	\begin{equation}\notag
		\mathcal{W}_A\defeq\left\{v\in D(\dive^\infty):\abs{v}\le 1\ \mass\text{-a.e.\ }\supp v\Subset A\right\}.
	\end{equation}
	Finally, the supremum can be equivalently taken among all $v\in\mathcal{\tilde{W}}_A$, where
	\begin{equation}\notag
		\mathcal{\tilde{W}}_A\defeq\left\{v\in D(\dive^\infty):\abs{v}\le 1\ \mass\text{-a.e.\ }\supp v\subseteq A\right\}.
	\end{equation}
\end{prop}

\begin{proof}
	Fix $A\subseteq\XX$ open.
	If $v\in\mathcal{W}_A$, as $\supp\,v\Subset A$, we can find $B$ open with $v\in\mathcal{W}_B$ and $\bar{B}\subseteq A$. Take now a sequence $\{f_n\}_n\subseteq\LIPbs(\XX)$ with $f_n\rightarrow f$ in $\Lpu$ and $(\lip(f_n)\mass)(\XX)\rightarrow\abs{\DIFF f}(\XX)$ (hence $\lip(f_n)\mass\rightharpoonup\abs{\DIFF f}$ in duality with $\Cb(\XX)$). Then $$\int_\XX f\dive\, v\dd{\mass}=\lim_n\int_\XX f_n\dive\, v\dd{\mass}=- \lim_n \int_\XX\diff{f_n}(v)\dd{\mass}. $$
	We have that for every $n$ (recall the bound $\abs{\diff f_n}\le\lip (f_n)\ \mass$-a.e.),
	$$ \abs{  \int_\XX\diff{f_n}(v)\dd{\mass}}\le \int_B \lip (f_n)\dd{\mass}.$$
	Exploiting the weak convergence $\lip(f_n)\mass\rightharpoonup\abs{\DIFF f}$ we have
	\begin{equation*}
		\limsup_n\int_B \lip (f_n)\dd{\mass}\le \abs{\DIFF f}(\bar{B})\le \abs{\DIFF f}(A)
	\end{equation*}
	and this proves that the quantity defined by the supremum in \eqref{intagainst} is bounded by $\abs{\DIFF f}(A)$.
	
	Now, (with the notation of \cite{DiMarino14,buffa2020bv}), let $\delta\in \mathrm{Der}^{\infty,\infty}(\XX)$ be with $\abs{\delta}\le 1\ \mass$-a.e.\ and $\supp\delta\Subset A$. Then $\delta\in \mathrm{Der}^{2,2}(\XX)$ so that, using \cite[Lemma 3.12]{buffa2020bv}, we can find a vector field $v_\delta\in D(\dive)$ such that $\abs{v_\delta}\le\abs{\delta}\ \mass$-a.e.\ and $\dive\,v_\delta=\dive\,\delta\ \mass$-a.e.\ and then also the opposite inequality in \eqref{intagainst} is proved, in virtue of \cite[Theorem 3.4]{DiMarino14}.
	
	In order to conclude, we just have to show that if $A\subseteq\XX$ is open and $v\in\mathcal{\tilde{W}}_A$, then $$ \int_\XX f\dive\, v\dd{\mass}\le\abs{\DIFF f}(A).$$
	By an immediate approximation argument, there is no loss of generality in assuming that $v$ has bounded support.
	Let $\epsilon>0$. By regularity, let $K\subseteq\XX$ be a compact set with $K\subseteq\XX\setminus A$ and $\abs{\DIFF f}((\XX\setminus A)\setminus K)<\epsilon$. It is clear that $\supp v\Subset\XX\setminus K$, so that 
	\begin{equation*}
		\int_\XX f\dive\,v\dd{\mass}\le \abs{\DIFF f}(\XX\setminus K)\le \abs{\DIFF f}(A)+\epsilon,
	\end{equation*}
	so that the proof is concluded being $\epsilon>0$ arbitrary.
\end{proof}

\begin{rem}\label{interpretationint}
	If $f\in\BVv$, $v\in D(\dive)\cap\Lpi$ and $\{n_k\}_k\subseteq(0,\infty)$, $\{m_k\}_k\subseteq(0,\infty)$ are two sequences with $\lim_k n_k=\lim_k m_k=+\infty$, then the limit
	\begin{equation}\label{defint}
		\lim_k  \int_\XX(f\vee -m_k)\wedge n_k\dive\,v\dd{\mass}
	\end{equation}
	exists finite and does not depend on the particular choice of the sequences $\{n_k\}_k$ and $\{m_k\}_k$.  
	Indeed, a cut-off argument and an approximation argument as the one in the proof of Proposition \ref{reprfordiffregularpre2} yields that, if $g\in\BVv\cap\Lpi$ and  $v$ is as above, then $$\abs{\int_\XX g \dive\,v\dd{\mass}}\le \abs{\DIFF g}(\XX)\Vert v\Vert_{\tanXinf},$$ so that, using also \eqref{coareaeqdiff}, we get the claim. 
	
	Therefore, if $f\in\BVv$ and $v\in D(\dive)\cap\Lpi$, we can write
	$$\int_\XX f\dive\, v\dd{\mass},$$
	with the convention that it has to be interpreted as the limit in \eqref{defint}.\fr
\end{rem}
\subsection{Total variation and capacity}
	We recall the definition of the $2$-capacity on a metric measure space $(\XX,\dist,\mass)$ (to which we shall simply refer as capacity): for any set $A\subseteq\XX$ we set 
$$\capa(A)\defeq\inf\left\{ \Vert f\Vert_{\HSs}^2:f\in\HSs,f\ge 1\ \mass\text{-a.e.\ on some neighbourhood of $A$}\right\}. $$
\begin{lem}\label{capaconlip}
Let $(\XX,\dist,\mass)$ be a metric measure space and let $K\subseteq\XX$ be a compact set.
Then 
\begin{equation}\label{capaineq}
	\capa(K)=\inf \Vert f\Vert^2_{\HSs}=\inf\int_\XX f^2+\lip(f)^2\dd{\mass}
\end{equation}
where both the infima are taken among all functions $f\in\LIPbs(\XX)$ such that $f\ge1$ on a neighbourhood of $K$. 
\end{lem}
\begin{proof}
Recall that, by its very definition, $\capa(K)=\inf \Vert f\Vert_{\HSs}^2$, where the infimum is taken among all functions $f\in\HSs$ such that $f\ge 1\ \mass$-a.e.\ on a neighbourhood of $K$.

Recalling that if $f\in\LIPbs(\XX)$, then $f\in\HSs$ and $$\Vert f\Vert_{\HSs}^2\le\int_\XX f^2+\lip(f)^2\dd{\mass},$$ we immediately obtain the two inequalities $(\le)$ in \eqref{capaineq}.

To conclude, we can assume with no loss of generality that $\capa(K)<\infty$. If $\epsilon>0$, fix $g\in\HSs$ with $g\ge 1\ \mass$-a.e.\ on a neighbourhood of $K$ such that $\Vert g \Vert_{\HSs}^2 \le \capa(K)+\epsilon$. Up to replacing $g$ with $0\vee g\wedge 1$, there is no loss of generality in assuming that $g$ takes values in $[0,1]$ and that $g= 1\ \mass$-a.e.\ on a neighbourhood of $K$, call this neighbourhood $A$. Let also $\{g_n\}\subseteq\LIPbs(\XX)$ be such that $g_n\rightarrow g$ in $\Lpt$ and $\int_\XX\lip(g_n)^2\dd{\mass}\rightarrow 2\Ch(g)$ (using an immediate cut-off argument we can replace $\LIPb(\XX)\cap\Lpt$ with $\LIPbs(\XX)$ in \eqref{defch}).
Take $\eta\in\LIPbs(\XX)$ such that $\eta=1$ on a neighbourhood of $K$, $\eta(x)\in[0,1]$ for every $x\in\XX$ and $\supp\eta\subseteq A$ (here we use the compactness of $K$).
Set now $f_n\defeq (1-\eta)g_n+\eta\in\LIPbs(\XX)$ and notice that $f_n\ge 1$ on a neighbourhood of $K$. Exploiting the fact that $g_n\rightarrow g$ in $\Lpt$ and $g=1\ \mass$-a.e.\ on $A$, $$\limsup_n\int_\XX f_n^2\dd{\mass} = \int_\XX g^2\dd{\mass}.$$ Using the convexity inequality for the slope (e.g.\ \cite[Lemma 1.3.2]{DiM14a}) and arguing as above, we have that
$$
\lip(f_n)\le (1-\eta) \lip(g_n)+\lip(\eta) \abs{g_n-1}
$$
so that
$$\limsup_n \int_\XX\lip(f_n)^2\dd{\mass}\le \limsup_n\int_\XX\lip(g_n)^2\dd{\mass}.$$
All in all, we conclude as $\epsilon>0$ was arbitrary and
\begin{equation*}\notag
\limsup_n\int_\XX f_n^2+ \lip(f_n)^2\dd{\mass}\le \Vert g\Vert_{\HSs}^2\le\capa(K)+\epsilon.\qedhere
\end{equation*}
\end{proof}
\begin{rem}
It is worth pointing out that Lemma \ref{capaconlip} holds also replacing $\lip$ with the bigger $\lip_a$ in \eqref{capaineq}, which is defined by 
$$
\lip_a(f)(x)\defeq\limsup_{y,z\rightarrow x}\frac{\abs{f(y)-f(z)}}{\dist(y,z)},
$$
which has to be understood to be $0$ if $x$ is an isolated point, for any $f$ locally Lipschitz. The proof is exactly the same, if one takes into account the main result of \cite{AmbrosioGigliSavare11-3}.
\fr
\end{rem}
In the framework of $\PI$ spaces, the fact that the total variation of a function of bounded variation is absolutely continuous with respect to the capacity is a consequence of \eqref{diffllxxxeq}. We prove here that this result holds even without any assumption on the m.m.s.
\begin{thm}\label{diffllcapthm}
Let $(\XX,\dist,\mass)$ be a metric measures space and $f\in\BVv$. Then
$$ 
\abs{\DIFF f}\ll\capa.
$$
\end{thm}
\begin{proof}
First notice that thanks to \eqref{coareaeqdiff} and the regularity of $\abs{\DIFF f}$, we can reduce ourselves to prove that $\abs{\DIFF f}(K)=0$ whenever $K\subseteq\XX$ is a compact set with $\capa(K)=0$ and assuming also $f\in\BVv\cap\Lpi$. Thanks to Lemma \ref{capaconlip}, we can take a sequence $\{\varphi_n\}_n\subseteq\LIPbs(\XX)$ such that $\varphi_n(x)\in[0,1]$ for every $x\in\XX$,  $\varphi_n(x)=1$ on a neighbourhood of $K$ (this neighbourhood depends on $n$) and $\Vert \varphi_n\Vert_{\HSs}\rightarrow 0$.

Take $v\in D(\dive^\infty)$ with $ \abs{v}\le 1\ \mass\text{-a.e.\ and\ } \supp v$ bounded. Consider now 
$${\int_\XX f\dive\, v \dd{\mass}}=\int_\XX f\dive(\varphi_ n v) \dd{\mass}+\int_\XX f\dive((1-\varphi_ n)v) \dd{\mass} $$
and notice that, by the calculus rules for the divergence in \eqref{calcdiveeq} (recall that we are assuming $f\in\Lpt$),
$$ \int_\XX f\dive(\varphi_ n v) \dd{\mass}\rightarrow 0\quad\text{as }n\rightarrow \infty$$
and also that, by Proposition \ref{reprfordiffregularpre2},
$$\abs{\int_\XX f\dive((1-\varphi_ n)v) \dd{\mass} }\le \abs{\DIFF f}(\XX\setminus K) $$
as $\supp ((1-\varphi_ n)v)\Subset \XX\setminus K. $
If we let $n\rightarrow\infty$ and then take the supremum among all $v$ as above, we have, by Proposition \ref{reprfordiffregularpre2},
$$\abs{\DIFF f}(\XX)\le  \abs{\DIFF f}(\XX\setminus K),$$
which proves our claim.
\end{proof}
\subsection{Vector valued functions of bounded variation}
In what follows we fix $n\in\NN$, $n\ge 1$.
We treat now the case of vector valued $\BV$ functions, i.e.\ functions of bounded variation taking values in $\RR^n$, or equivalently, collections of $n$ real valued functions of bounded variation. As the case $n=1$ has already been treated, we focus on $n\in\NN$, $n\ge 2$.
\begin{defn}\label{def:defBVv}
	Let $(\XX,\dist,\mass)$ be a metric measure space and $F\in\Lpu^n$.
	We define, for any $A$ open subset of $\XX$,
	\begin{equation}\label{defntvvector}
		\abs{\DIFF F}(A)\defeq \inf \left\{\liminf_k \int_A \Vert (\lip(F_{i,k}))_{i=1,\dots,n}\Vert_e\dd{\mass}\right\}
	\end{equation}
	where the infimum is taken among all sequences $\{F_{i,k}\}_k\subseteq\LIPloc(A)$ such that $F_{i,k}\rightarrow F_i$ in $\Lp^1(A,\mass)$ for every $i=1,\dots,n$.  If $\abs{\DIFF F }(\XX)<\infty$, we say that $F\in\BV(\XX)^n$.
\end{defn}
\begin{rem}\label{whattorelax}
	Notice that we are relaxing the integral of the Euclidean norm of the vector whose components are the local Lipschitz constants of the various coordinates, not the local Lipschitz constant of a vector valued function. The former approach follows \cite{AFP00}, while the latter (a slight variant of the one in) \cite{MIRANDA2003}. For open subsets of $\RR^d$ the former approach corresponds to the relaxation of the integral of the Hilbert-Schmidt norm  of the Jacobian matrix of a sequence of approximating functions, while the latter employs the operator norm instead, and is seen to be equivalent to the one proposed in \cite{ambmetric}. Also, it is straightforward to show that $F\in\BVv^n$ if and only if the quantity defined in \eqref{defntvvector} for $\abs{\DIFF F}(\XX)$ is finite.\fr
\end{rem}
\begin{prop}
	Let $(\XX,\dist,\mass)$ be a metric measure space and $F\in\BVv^n$. Then $ \abs{\DIFF F}(\,\cdot\,)$ as defined in \eqref{defntvvector} is the restriction to open sets of a finite non negative Borel measure that we call total variation of $F$ and still denote with the same symbol.
\end{prop}
\begin{proof}
	The proof of \cite[Lemma 5.2]{ADM2014} can be easily adapted with no substantial changes. 
	Indeed, one has only to notice that the convexity inequality for the slope used in \cite[Lemma 5.4]{ADM2014} and the properties of the Euclidean norm imply a suitable version of the convexity inequality for the slope in our situation. 
\end{proof}

\section{The theory for RCD spaces}
\subsection{Some useful results}
	The proof of the following result can be found in \cite[Remark 3.5]{GigliHan14}, we briefly sketch it here for the sake of completeness.
\begin{prop}[Bakry-\'{E}mery estimate in BV]\label{bakryemry}
	Let $(\XX,\dist,\mass)$ be an $\RCD(K,\infty)$ space and $f\in\BVv$. Then, if $t>0$, $\heat_t f\in\BVv$ and it holds
	\begin{equation}\notag
		\abs{\DIFF \heat_t f}\le e^{-K t}\heat_t\abs{\DIFF f} .
	\end{equation}
	If moreover $f\in\BVv\cap\Lpi$, then $\heat_t f\in\BVv\cap\HSs$ and 
	\begin{equation}\notag
		\abs{\nabla \heat_t f}\le e^{-K t}\heat_t\abs{\DIFF f}\quad\mass\text{-a.e.}
	\end{equation}
\end{prop}
\begin{proof}
	First notice that by the general theory of Sobolev spaces, we easily obtain that $\abs{\DIFF \heat_t f}\le\abs{\nabla \heat_t f}\mass$ if $f\in\BVv\cap\Lpi$.
	Then, thanks to the lower semicontinuity of the total variation and a truncation argument, the first statement follows from the second.
	
	In order to conclude, take a sequence $\{f_k\}_k\subseteq\LIPbs(\XX)$ with $f_k\rightarrow f$ in $\Lpu$ and $(\lip(f_k)\mass)(\XX)\rightarrow\abs{\DIFF f}(\XX)$ (hence $\lip(f_k)\mass\rightharpoonup\abs{\DIFF f}$ in duality with $\Cb(\XX)$). Clearly, we can assume that $\Vert f_k\Vert_{\Lpi}\le\Vert f\Vert_{\Lpi}$, so that $\heat_t f_k\in\HSs\cap\LIPb(\XX)$ for every $k$, with equibounded Lipschitz constants, by the $\Lp^\infty$-$\LIP$ regularization property.
	Also, by \cite[Corollary 4.3]{Savare13}, we have that for every $k$,
	$$ \abs{\nabla \heat_t f_k}\le e^{-K t}\heat_t\abs{\nabla f_k}\le e^{-K t}\heat_t \lip(f_k)\quad\mass\text{-a.e.}$$
	Then, $\{ \abs{\nabla \heat_t f_k}\}_k\subseteq\Lpt$ is bounded and, as $\abs{\nabla \heat_t f}$ is bounded from above by any $\Lp^2$ weak limit of $\{\abs{\nabla \heat_t f_k}\}_k$, we can conclude easily, recalling that the heat flow on finite measures preserves the weak convergence in duality with $\Cb(\XX)$.
\end{proof}

For the rest of this subsection, we fix $n\in\NN$, $n\ge 1$. The following proposition provides us with a generalization of Proposition \ref{reprfordiffregularpre2} to the multi dimensional case in the context of $\RCD(K,\infty)$ spaces.  First, we need an approximation lemma. Recall also our definition of $\TestV(\XX)$ in \eqref{testVdef}.
\begin{lem}\label{approxHod}
	Let $(\XX,\dist,\mass)$ be an $\RCD(K,\infty)$ space and let $v=(v_1,\dots,v_n)\in \WHHSsn$ with $\abs{v}\le 1\ \mass$-a.e. Then there exists a sequence $\{v^k=(v_1^k,\dots,v_n^k)\}_k\subseteq\TestV(\XX)^n$ such that $|{v^k}|\le  1\ \mass$-a.e.\ for every $k$ and $v^k_i\rightarrow v_i$ in $\WHHSs$ for every $i=1,\dots,n$.
\end{lem}
\begin{proof}
	By the very definition of $\WHHSs$, for every $i=1,\dots,n$, we have a sequence $\{w^k_i\}\subseteq\TestV(\XX)^n$ with $w^k_i\rightarrow v_i$ in $\WHHSs$.
	Set then, for $\epsilon>0$, $$v^{k,\epsilon}_i\defeq \frac{1}{(1+\epsilon)\vee\sqrt{\sum_j |{w^k_j}|^2}} w^k_i.$$
	and finally  $v^{k,\epsilon}\defeq (v^{k,\epsilon}_1,\dots,v^{k,\epsilon}_n)$. It is clear that $|{v^{k,\epsilon}}|\le 1\ \mass$-a.e.\ so that, using also a diagonal argument, it suffices to show that for every $i=1,\dots,n$, 
	$$  v^{k,\epsilon}_i\rightarrow \frac{1}{1+\epsilon} v_i\quad\text{in }\WHHSs\text{ as }k\rightarrow\infty.$$
	
	Fix then $i=1,\dots,n$ and $\epsilon>0$. It is clear that $$ v^{k,\epsilon}_i\rightarrow \frac{1}{1+\epsilon} v_i\quad\text{in }\tanX\text{ as }k\rightarrow\infty.$$ By the calculus rules in Lemma \ref{calculushodge} below, we just have to show that 
	\begin{equation}\notag
		\abs{\nabla\frac{1}{(1+\epsilon)\vee\sqrt{\sum_j |{w^k_j}|^2}} }|{w^k_i}|\rightarrow0\quad\text{in }\Lpt\text{ as }k\rightarrow\infty.
	\end{equation}
	Set now $A^{k,\epsilon}\defeq\left\{\sqrt{\sum_j |{w^k_j}|^2}>1+\epsilon \right\}$ and notice that $A^{k,\epsilon}\rightarrow \emptyset$ in $\Lpo$ as $k\rightarrow\infty$. 
	Using the calculus rules, we can estimate
	\begin{equation}\notag
		\begin{split}
			&\abs{\nabla\frac{1}{(1+\epsilon)\vee\sqrt{\sum_j |{w^k_j}|^2}} }|{w^k_i}|\\
			&\qquad\le\frac{1}{2}\chi_{A^{k,\epsilon}}\left(\frac{1}{\sum_j |{w^k_j}|^2}\right)^{3/2}\bigg|{\nabla \sum_j |{w^k_j}|^2}\bigg|
			|{w^k_i}|\\&\qquad\le 	\chi_{A^{k,\epsilon}}\left(\frac{1}{\sum_j |{w^k_j}|^2}\right)^{3/2} \left( \sum_j|{w^k_j}|^2\right)^{1/2}\left( \sum_j|\nabla{w^k_j}|^2\right)^{1/2} |{w^k_i}|
			\\&\qquad\le\chi_{A^{k,\epsilon}} \left(\frac{1}{1+\epsilon}\right)^2	 \sum_j|\nabla{w^k_j}||{w^k_i}|,
		\end{split}
	\end{equation}
	where in the second inequality we used the Cauchy-Schwarz inequality,
	and then we see that the last term converges to $0$ in $\Lpt$ as $k\rightarrow\infty$.
\end{proof}
In the previous proof we used the following calculus rules, which are an immediate consequence of the already known ones proved in \cite{Gigli14}. We add also another lemma, again based on \cite{Gigli14}, which is not explicitly used in this work but whose proof grants coincidence between the definitions of $\WHCSs$ via the usual definition of test vector fields and our definition of test vector field.
	\begin{lem}\label{calculushodge}
	Let $(\XX,\dist,\mass)$ be an $\RCD(K,\infty)$ space, $X\in\WSHs\cap\tanXinf$ and $f\in \mathrm{S}^2(\XX)\cap\Lpi$. Then $f X\in \WSHs$ and 
	\begin{alignat}{2}\notag
		\dive (f X)&=\nabla f\,\cdot\, X+f\dive\, X,\\\notag
		\diff (f X)&=\nabla f\wedge X+f\diff X.
	\end{alignat}
	If moreover $X\in\WHHSs$, then $f X\in \WHHSs$.
\end{lem}
\begin{proof}
	Recall that $\WSHs$ and $\WHHSs$ are Hilbert spaces. We prove the claim with an approximation argument. Also, as in the discussion after \cite[Definition 3.5.11]{Gigli14}, if $\omega \in\tanX$, then $\omega\in D(\delta)$ if and only if $\omega\in D(\dive)$ and, if this is the case, $\delta\omega=-\dive(\omega)$.
	
	If $f\in\TestF(\XX)$, the claim is a consequence of \cite[Proposition 3.5.4]{Gigli14} (which is stated with a slightly different definition of $\TestV(\XX)$) and the calculus rules for the divergence and the following approximation argument. If $\{X_n\}_n\subseteq\TestV(\XX)$ with $X_n\rightarrow X$ in $\WSHs$, then $f X_n\in\TestV(\XX)$ and $f X_n\rightarrow f X$ in $\WSHs$ (see the next paragraph for more details).
	
	If $f\in\HSs\cap\Lpi$, take $\{f_n\}_n\subseteq\TestF(\XX)$ with $\Vert f_n \Vert_{\Lpi}$ uniformly bounded and $f_n\rightarrow f$ in $\HSs$. Now we can use the calculus rules for $f_n\in\TestF(\XX)$ and easily prove, using also dominated convergence,
	\begin{alignat}{4}\notag
		\dive(f_n X)&=\nabla f_n\,\cdot\,X+f_n\dive\,X\rightarrow \nabla f\,\cdot\,X+f\dive\,X\quad&&\text{in }\Lpt\\\notag
		\diff (f_n X)&=\nabla f_n\wedge X+ f_n\diff X\rightarrow\nabla f\wedge X+ f\diff X\quad&&\text{in }\mathrm{L}^2(\Lambda^2 T\XX).
	\end{alignat}
	This shows that $f X\in\WSHs$, that the calculus rules hold and finally that $f_n X\rightarrow f X$ in $\WSHs$, so that if $X\in\WHHSs$, then $f X\in \WHHSs$.
	
	If $f\in\Ss\cap\Lpi$ we fix $\bar{x}\in\XX$ and we take $\{\phi_n\}_n\subseteq\LIPbs(\XX)$ with $\phi_n(x)\defeq ((n-\dist(x,\bar{x}))\wedge 1)^+$. Similar computations to the ones of the previous paragraph with $f \phi_n$ in place of $f_n$ show that $f X\in\WSHs$, that the calculus rules hold and finally that $(f\phi_n) X\rightarrow f X$ in $\WSHs$, so that if $X\in\WHHSs$, then $f X\in \WHHSs$.
\end{proof}

\begin{lem}\label{calculuscova}
	Let $(\XX,\dist,\mass)$ be an $\RCD(K,\infty)$ space, $X\in\WSCs\cap\tanXinf$ and $f\in \mathrm{S}^2(\XX)\cap\Lpi$. Then $f X\in \WSCs$ and 
	\begin{equation}\notag
		\nabla (f X)=\nabla f\otimes X+f\nabla X.
	\end{equation}
	If moreover $X\in\WHCSs$, then $f X\in \WHCSs$.
\end{lem}
\begin{proof}
	Recall that $\WSCs$ and $\WHCSs$ are Hilbert spaces. We prove the claim with an approximation argument.

	Assume first $f\in \HSs\cap\Lpi$. Then the first part of the statement has been proved in \cite[Proposition 3.4.5]{Gigli14}. The second part follows approximating $X$ with a sequence of test vector fields.
	
	If $f\in\Ss\cap\Lpi$ we fix $\bar{x}\in\XX$ and we take $\{\phi_n\}_n\subseteq\LIPbs(\XX)$ with $\phi_n(x)\defeq ((n-\dist(x,\bar{x}))\wedge 1)^+$ and we set $f_n\defeq f\phi_n$. We can use the calculus rule for $f_n\in\HSs\cap\Lpi$ and easily prove, using also dominated convergence,
	$$ 
	\nabla (f_n X)=\nabla f_n\otimes X+f_n\nabla X\rightarrow\nabla f\otimes X+f\nabla X\quad\text{in }\mathrm{L}^2(T^{\otimes 2}\XX).
	$$
	This shows that $f X\in\WSCs$, that the calculus rule holds and finally that $f_n X\rightarrow f X$ in $\WSCs$, so that if $X\in\WHCSs$, then $f X\in \WHCSs$.
\end{proof}

In view of the following proposition, recall that the interpretation of the integral in \eqref{intbypartsvet} is given by Remark \ref{interpretationint}. 
\begin{prop}\label{reprvett1}
	Let $(\XX,\dist,\mass)$ be an $\RCD(K,\infty)$ space and $F\in\BVv^n$. Then, for every $A$ open subset of $\XX$, it holds that 
	\begin{equation}
		\label{intbypartsvet}
		\abs{\DIFF F}(A)=\sup\left\{\sum_{i=1}^n\int_A F_i \dive\, v_i\dd{\mass}\right\},
	\end{equation}
	where the supremum is taken among all $v=(v_1,\dots,v_n)\in\mathcal{W}_A^n$, where
	
	\begin{equation}
		\label{mathvvet}
		\begin{split}
			\mathcal{W}_A^n\defeq\Big\{v=(v_1,\dots v_n)\in\TestV(\XX)^n:\abs{v}\le 1\ \mass\text{-a.e.\ } \supp \abs{v}\Subset A\Big\}.
		\end{split}
	\end{equation}
	Finally, the supremum can be equivalently taken among all $\mathcal{\tilde{W}}_A^n$
	\begin{equation}
		\notag
		\begin{split}
			\mathcal{\tilde{W}}_A^n\defeq\Big\{v=(v_1,\dots v_n)\in\TestV(\XX)^n:\abs{v}\le 1\ \mass\text{-a.e.\ } \supp \abs{v}\subseteq A\Big\}.
		\end{split}
	\end{equation}
\end{prop}

\begin{proof}
	Call $\abs{\DIFF F}^*$ the quantity defined by the right hand side of \eqref{intbypartsvet}, we show now that $\abs{\DIFF F}^*$ is the restriction to open sets of a finite Borel measure, that we still denote with $\abs{\DIFF F}^*$ and that $\abs{\DIFF F}^*=\abs{\DIFF F}$ as measures.
	\\\textsc{Step 1}.
	We show that $\abs{\DIFF F}^*(A)\le \abs{\DIFF F}(A)$ for every open set $A$. Fix then $A\subseteq\XX$ open. Assume for the moment that also
	$F_i\in\Lp^\infty(A,\mass)\cap\LIPloc(A)$ is such that $\int_A\lip(F_i)\dd{\mass}<\infty$ for every $i=1,\dots,n$ and take any $v=(v_1,\dots,v_n)\in\WW_A^n$. Set now $$C\defeq\supp\abs{v}$$ Notice $C\Subset A$ and take a cut-off function $\psi\in\LIPbs(\XX)$ with $\psi(x)\in[0,1]$ for every $x\in\XX$, $\psi=1$ on a neighbourhood of $C$ and $\supp\psi\Subset A$. Therefore, for every $i=1,\dots,n$, $\psi F_i\in\Lp^\infty(\mass)\cap\LIPloc(\XX)$ is such that $\int_\XX\lip(\psi F_i)\dd{\mass}<\infty$. We can now estimate, for $t>0$, using the Cauchy-Schwarz inequality and Proposition \ref{bakryemry},
	\begin{equation}\notag
		\begin{split}
			-\sum_{i=1}^n\int \heat_t(\psi F_i) \dive\, v_i\dd{\mass}&=
			\sum_{i=1}^n\int \nabla \heat_t(\psi F_i)\,\cdot\, v_i\dd{\mass}\le \int_C \Vert (\abs{\nabla \heat_t(\psi F_i)})_i\Vert_e \dd{\mass}\\&\le e^{-K t} \int_C \Vert ({ \heat_t\abs{\DIFF ( \psi F_i)}})_i\Vert_e \dd{\mass}\le e^{-K t} \int_C \Vert ({ \heat_t{\lip ( \psi F_i)}})_i\Vert_e \dd{\mass}
		\end{split}
	\end{equation}
	so that, letting $t\searrow  0$,
	$$ -\sum_{i=1}^n\int  F_i \dive\, v_i\dd{\mass}\le \int_C \Vert( \lip(\psi F_i))_i\Vert_e\le \int_A \Vert \lip(F_i)\Vert_e.$$
	
	Back to the general case $F\in\BVv^n$, we notice that we have to show the claim in the case $F_i\in\Lp^\infty(A,\mass)$ for every $i=1,\dots,n$. Then, we can conclude by the very definition of $\abs{\DIFF F}$ and what said above, noticing that approximating sequences can be taken made of functions equibounded in $\Lpi$ with no loss of generality.
	\\\textsc{Step 2}. We show that $\abs{\DIFF F}^*$ is the restriction to open sets of a finite Borel measure (that we still call $\abs{\DIFF F}^*$).
	To this aim, we can use Carathéodory criterion (\cite{AFP00}, cf.\ \cite[Proof of Lemma 5.2]{ADM2014}) and is then enough to verify (all the sets in consideration are assumed to be open):
	\begin{enumerate}
		\item $\abs{\DIFF F}^*(A)\le\abs{\DIFF F}^*(B)$ if $A\subseteq B$,
		\item $\abs{\DIFF F}^*(A\cup B)\ge\abs{\DIFF F}^*(A)+\abs{\DIFF F}^*(B)$ if $\dist(A,B)>0$,
		\item $\abs{\DIFF F}^*(A)=\lim_k\abs{\DIFF F}^*(A_k)$ if $A_k\nearrow A$,
		\item$\abs{\DIFF F}^*(A\cup B)\le \abs{\DIFF F}^*(A)+\abs{\DIFF F}^*(B)$.
	\end{enumerate}
	We notice that $(1)$ and $(2)$ follow trivially from the definition of $\abs{\DIFF F}^*$ and that $(2)$ does not even need the sets to be well separated. We prove now property $(3)$. Fix $\epsilon>0$ and take a compact subset $K$ with $K\subseteq A$ and $\abs{\DIFF F}(A\setminus K)\le\epsilon$. Then there exists $\bar{k}$ such that $K\subseteq A_{\bar{k}}$, in particular we can find $\psi\in\LIPbs(\XX)$ with $\psi(x)\in[0,1]$ for every $x\in\XX$, $\psi=1$ on a neighbourhood of $K$ and $\supp\psi\Subset A_{\bar{k}}$. If we take $v=(v_1,\dots,v_n)\in\WW^n_A$, we can write $v_i=\psi v_i+(1-\psi)v_i$ for $i=1,\dots,n$, notice $(\psi v_i)_i \in\WW^n_{A_{\bar{k}}}$ and $((1-\psi)v)_i\in \WW^n_{A\setminus K}$. Then we can compute, using that $\abs{\DIFF F}^*(A\setminus K)\le \abs{\DIFF F}(A\setminus K)\le\epsilon$,
	$$\sum_{i=1}^n \int_A F_i\dive\, v_i\dd{\mass}=\sum_{i=1}^n \int_{A_{\bar{k}}} F_i\dive (\psi v_i)\dd{\mass}+\sum_{i=1}^n \int_{A\setminus K} F_i\dive ((1-\psi) v_i)\dd{\mass}\le \abs{\DIFF F}^*(A_{\bar{k}})+\epsilon
	$$
	so that $(3)$ follows as $v\in\WW^n_A$ and $\epsilon>0$ are arbitrary.
		We prove now $(4)$. Take a sequence of bounded open sets $\{A_k\}_k$ with $A_k\nearrow A$ and $A_k\subseteq \left\{x\in A:\dist(x,\XX\setminus A)>k^{-1}\right\}$; take similarly $\{B_k\}_k$. Fix $k$ and take $\tilde{\psi}_A\in\LIPbs(\XX)$ with $\tilde{\psi}_A(x)\in[0,1]$ for every $x\in\XX$, $\tilde{\psi}_A=1$ on a neighbourhood of $A_k$ and $\supp\tilde{\psi}_A\Subset A$; define similarly $\tilde{\psi}_B$. Define also ${\psi}_A\defeq\tilde{\psi}_A$ and ${\psi}_B\defeq\tilde{\psi}_B(1-\tilde{\psi}_A)$.
		Take then $v=(v_1,\dots,v_n)\in\WW^n_{A_k\cup B_k}$.  
		Writing $v_i={\psi}_A v_i+{\psi}_B v_i$ for $i=1,\dots,n$ we can argue similarly as above to verify that
		$$ \abs{\DIFF F}^*(A_k\cup B_k)\le \abs{\DIFF F}^*(A)+\abs{\DIFF F}^*(B)$$
		so that $(4)$ follows letting $k\rightarrow\infty$, taking into account $(3)$.
	\\\textsc{Step 3}.	We conclude that $\abs{\DIFF F}^*= \abs{\DIFF F}$. By the previous steps, it is enough to show $\abs{\DIFF F}^*(\XX)\ge \abs{\DIFF F}(A)$ if $A\subseteq\XX$ is open and bounded.
	Assume for the moment that also $F\in\Lpi^n$. Let $\{t_k\}_k$, $t_k\searrow 0$ and consider $F_{i,k}\defeq \heat_{t_k}F_i$. By lower semicontinuity of the total variation,
	\begin{equation}\notag
		\abs{\DIFF F}(A)\le \liminf_k \abs{\DIFF (F_{1,k},\dots,F_{n,k})}(A)
	\end{equation}
	and then, taking into account that $\mass(A)<\infty$ and the general theory of Sobolev spaces,
	\begin{equation}\label{toapprox}
		\abs{\DIFF F}(A)\le \liminf_k \int_A \Vert (\abs{\nabla F_{i,k}})_{i=1,\dots,n}\Vert_e\dd{\mass}\le \liminf_k \int_\XX \Vert (\abs{\nabla F_{i,k}})_{i=1,\dots,n}\Vert_e\dd{\mass}.
	\end{equation}
	By density, we take $F_{i,k,l}\subseteq\TestF(\XX)$ such that $F_{i,k,l}\rightarrow F_{i,k}$ in $\HSs$ as $l\rightarrow\infty$ for every $i$.
	We can write the right hand side of \eqref{toapprox} as
	\begin{equation}\notag
		\liminf_k\lim_{\epsilon\searrow 0}\lim_l\lim_{\delta \searrow 0}\sum_i \int_\XX \nabla F_{i,k} \cdot\frac{\nabla \heat_\delta F_{i,k,l}}{\sqrt{\sum_j \heat_\delta(\abs{\nabla F_{j,k,l}}^2)+\epsilon}}\dd{\mass},
	\end{equation}
	that is,
	\begin{equation}\label{quantity}
		\liminf_k\lim_{\epsilon\searrow 0}\lim_l\lim_{\delta \searrow 0}\sum_i \int_\XX   F_{i,k} \dive\left(\frac{\nabla \heat_\delta F_{i,k,l}}{\sqrt{\sum_j \heat_\delta(\abs{\nabla F_{j,k,l}}^2)+\epsilon}}\right)\dd{\mass}.
	\end{equation}
	Recalling the properties of the heat flow $\heat_{\mathrm{H},t}$, we can rewrite the quantity in \eqref{quantity} as
	\begin{equation}\notag
		\liminf_k\lim_{\epsilon\searrow 0}\lim_l\lim_{\delta \searrow 0}\sum_i \int_\XX   F_i \dive\left(\heat_{\mathrm{H,t_k}}\left(\frac{\nabla \heat_\delta F_{i,k,l}}{\sqrt{\sum_j \heat_\delta(\abs{\nabla F_{j,k,l}}^2)+\epsilon}}\right)\right)\dd{\mass}
	\end{equation}
	and see that it is bounded by $\abs{\DIFF F}^*(\XX)$, by an approximation argument that relies on Lemma \ref{approxHod}; here we used that an immediate approximation argument yields that if $A=\XX$  the request that $\supp v_i$ is compact in \eqref{mathvvet} is irrelevant. We have therefore proved $\abs{\DIFF F}^*(\XX)= \abs{\DIFF F}(\XX)$ in the case $F$ bounded. 
	
	We treat the general case. We write
	\begin{equation}\label{approxsucc}
		F^m\defeq((F_1\vee -m)\wedge m,\dots,(F_n\vee -m)\wedge m).
	\end{equation} 
	Now we can conclude easily, as, by lower semicontinuity, what we just proved and \eqref{coareaeqdiff}
	\begin{equation*}\notag
		\begin{split}
			\abs{\DIFF F}(\XX)&\le \liminf_m \abs{\DIFF (F^m)}(\XX)=	\abs{{\DIFF (F^m)}}^*(\XX)\le \abs{\DIFF F}^*(\XX)+\abs{\DIFF (F^m-F)}^*(\XX)\\&\le  \abs{\DIFF F}^*(\XX)+\sum_i \abs{\DIFF (F_i^m-F_i)}(\XX)\rightarrow \abs{\DIFF F}^*(\XX).
		\end{split}
	\end{equation*}
	
	The last claim can be proved as for Proposition \ref{reprfordiffregularpre2}.
\end{proof}

\begin{rem}
	One may wonder whether Proposition \ref{reprvett1} holds also in the more general setting of (infinitesimally Hilbertian) metric measure spaces, with the obvious modifications (i.e.\ whether we can extend Proposition \ref{reprfordiffregularpre2} to functions taking values in $\RR^n$ instead of $\RR$). It seems anything but straightforward to adapt the argument used in \cite{DiM14a} (extracted from \cite{DiMarino14,ADM2014}) as here we face a difficulty generalizing the approach via test plans. For this reason we had to provide a completely different proof, obtained via approximation arguments, at the price of working in more regular spaces.
	We give here an example of this issue, using the notation of the articles just cited. We point out that the difference $\abs{\DIFF f}\ne \abs{\DIFF f}_w$ that we are going to see is what we expect, given the choice of the relaxation made to define the total variation, cf.\ Remark \ref{whattorelax}.
	
	Consider $\XX\defeq [0,1]^2\subseteq\RR^2$ endowed with the Euclidean distance and the Lebesgue measure. Let $f:\XX\rightarrow\RR^2$ be the identity. It is clear that $f\in\BVv^2$ and $\abs{\DIFF f}(\XX)=\sqrt{2}$. However, computing the total variation defined via test plans, $\abs{\DIFF f}_w=1$.
	Indeed, if $B\subseteq\XX$ is a Borel set and $\boldsymbol{\pi}$ is a test plan, we obtain, using Fubini's theorem,
	\begin{equation}\notag
		\begin{split}
			\int \gamma_{\#} \abs{\DIFF( f\circ\gamma)}(B)\dd{\boldsymbol{\pi}(\gamma)}&\le \int \LL^1(\{t:\gamma_t\in B\})  \Lip(\gamma) \dd{\boldsymbol{\pi}(\gamma)}\\&\le \Vert\Lip(\gamma)\Vert_{\Lp^\infty(\boldsymbol{\pi})} (\LL^1\otimes \boldsymbol{\pi})(\{(t,\gamma):\gamma_t\in B\})\\&\le \Vert\Lip(\gamma)\Vert_{\Lp^\infty(\boldsymbol{\pi})} C(\boldsymbol{\pi})\LL^2(B),
		\end{split}
	\end{equation}
	so that $\abs{\DIFF f}_w\le \LL^2$.
	\fr
\end{rem}

\subsection{Fine modules}
In this subsection, we mostly recall the results of \cite{debin2019quasicontinuous}, which will be of great importance in what follows. 
\begin{thm}[{\cite[Theorem 2.6]{debin2019quasicontinuous}}]\label{tancapa}
	Let $(\XX,\dist,\mass)$ be an $\RCD(K,\infty)$ space.
	Then there exists a unique couple $(\tanXcap,\nablatilde)$, where $\tanXcap$ is a $\Lp^0(\capa)$-normed $\Lp^0(\capa)$-module and $\nablatilde:	\TestF(\XX) \rightarrow\tanXcap$ is a linear operator such that:
	\begin{enumerate}[label=\roman*)]
		\item
		$|{\nablatilde f}|=\qcr(\abs{\nabla f}) \ \capa$-a.e.\ for every $f\in\TestF(\XX)$,
		\item
		the set $\left\{\sum_{n} \chi_{E_n}\nablatilde f_n\right\}$, where $\{f_n\}_n\subseteq\TestF(\XX)$ and $\{E_n\}_n$ is a Borel partition of $\XX$ is dense in $\tanXcap$.
	\end{enumerate}
	Uniqueness is intended up to unique isomorphism, this is to say that if another couple $(\tanXcap',\nablatilde')$ satisfies the same properties, then there exists a unique module isomorphism $\Phi:\tanXcap\rightarrow\tanXcap'$ such that $\Phi\circ \nablatilde=\nablatilde'$.
	Moreover, $\tanXcap$ is a Hilbert module that we call capacitary tangent module.
\end{thm}
It is worth spending a few words on $\Lpc$-normed $\Lpc$-modules and, in particular, on $\tanXcap$, as the $\Lpc$ and $\Lpo$ topologies may behave quite differently. First, $\Lpc$-normed $\Lpc$-modules enjoy the following important properties (cf.\ \cite[Definition 1.2.1]{Gigli14}):
\begin{enumerate}[label=\roman*)]
	\item \textsc{locality}: for every $v\in\tanXcap$ and $\{A_i\}_{i}$ sequence of Borel subsets of $\XX$ such that $\chi_{A_i}v=0$ for every $i\in\NN$, then $\chi_{\bigcup_i A_i} v=0$,
	\item \textsc{gluing}: if $\{v_i\}_i\subseteq \tanXcap$ and $\{A_i\}_{i}$ is a sequence of pairwise disjoint Borel subsets of $\XX$, there exists $v\in\tanXcap$ such that $\chi_{A_i} v=\chi_{A_i} v_i$ for every $i\in\NN$.
\end{enumerate}
Indeed, the first property follows trivially from the existence of the pointwise norm. For what concerns the second property, notice first that, partitioning the sets $A_i$ and using locality, we can with no loss of generality assume that $\abs{v_i}\in\Lp^{\infty}(\capa)$ for every $i$. We can then set $a_i\defeq 2^{-i}\Vert 1+\abs{v_i}\Vert_{\Lp^\infty(\capa)}$ and consider the Cauchy sequence $n\mapsto \sum_{i=1}^n a_i^{-1}\chi_{A_i} v_i$ and then multiply its limit by $f\defeq \sum_{i=1}^\infty a_i\chi_{A_i} $ so that we can conclude by locality. However, the gluing property for $\tanXcap$ follows directly from its construction, starting from the set of infinite linear combinations as in item $ii)$ of Theorem \ref{tancapa}. Notice that one needs the the gluing property for $\tanXcap$ to define the multiplication by functions in $\Lpc$ so that we can not use the argument above to prove the gluing property for $\tanXcap$. This discussion is relevant because the map $$\Lpc\times\tanXcap\ni(f,v)\mapsto f v\in\tanXcap$$ is not continuous in general.
For example, set $(\XX,\dist,\mass)=([0,1],\dist_e,\LL^1)$, recall \cite[Example 2.17]{debin2019quasicontinuous} and notice that $\Lp^\infty(\capa)$ is a closed (non trivial) subspace of $\Lpc$. Take $v_n\defeq(1+n^{-1})\chi_{(0,1)}$ and $f(x)\defeq 1/x$. Clearly $v_n\rightarrow v\defeq \chi_{(0,1)}\in\tanXcap$, however $\{f v_n\}\in\tanXcap$ is not even a Cauchy sequence.

Notice that we can, and will, extend the map $\qcr$ from $\HSs$ to $\Ss\cap\Lpi$ by a locality argument.
We define
\begin{equation}\notag
	\TestVbar(\XX)\defeq \left\{ \sum_{i=1}^n \qcr(f_i) \nablatilde g_i :f_i\in\Ss\cap\Lpi,g_i\in\TestF(\XX) \right\}.
\end{equation}

We define also the vector subspace of quasi-continuous vector fields, $\Cqcvf$, as the closure of $\TestVbar(\XX)$ in $\tanXcap$ and finally,
\begin{equation}\label{Cqcvfdef}
	\Cqcvfinf\defeq\left\{v\in\Cqcvf:\abs{v}\text{ is $\capa$-essentially bounded} \right\}.
\end{equation}

Recall now that as $\mass\ll\capa$, we have a natural projection map
\begin{equation}\notag
	\Pr:\Lpc\rightarrow\Lpo \quad \text{defined as}\quad[f]_{\Lpc}\mapsto [f]_{\Lpo}
\end{equation}
where $[f]_{\Lpc}$ (resp.\ $[f]_{\Lpo}$) denotes the $\capa$ (resp.\ $\mass$) equivalence class of $f$. It turns out that $\Pr$, restricted to the set of quasi-continuous functions, is injective (\cite[Proposition 1.18]{debin2019quasicontinuous}).
We have the following projection map $\Prbar$, given by \cite[Proposition 2.9 and Proposition 2.13]{debin2019quasicontinuous}, which plays the role of $\Pr$ on vector fields. 
\begin{prop}\label{prbardef}
	Let $(\XX,\dist,\mass)$ be an $\RCD(K,\infty)$ space. There exists a unique linear continuous map \begin{equation}\notag
		\Prbar :\tanXcap\rightarrow\tanXzero
	\end{equation}
	that satisfies
	\begin{enumerate}[label=\roman*)]
		\item $\Prbar (\nablatilde f)=\nabla f$ for every $f\in\TestF(\XX)$,
		\item $\Prbar (g v)=\Pr(g)\Prbar(v)$ for every $g\in\Lpc$ and $v\in\tanXcap$.
	\end{enumerate}
	Moreover, for every $v\in\tanXcap$,
	\begin{equation}\notag
		\abs{\Prbar(v)}=\Pr(\abs{v})\quad\mass\text{-a.e.}
	\end{equation}
	and $\Prbar$, when restricted to the set of quasi-continuous vector fields, is injective.
\end{prop}
We point out that if $v\in\Cqcvf$, \cite[Proposition 2.12]{debin2019quasicontinuous} shows that $\abs{v}\in\Lpc$ is quasi-continuous, in particular,  $v\in \Cqcvfinf$ if and only if $\Prbar(v)\in\tanXinf$.

In what follows, with a little abuse, we often write, for $v\in\tanXcap$, $v\in D(\dive)$ if and only if $\Prbar (v)\in D(\dive)$ and, if this is the case, $\dive\, v=\dive(\Prbar(v))$. Similar notation will be used for other operators acting on subspaces of $\tanXzero$. 
\begin{thm}[{\cite[Theorem 2.14 and Proposition 2.13]{debin2019quasicontinuous}}]
	Let $(\XX,\dist,\mass)$ be an $\RCD(K,\infty)$ space. Then there exists a unique map $\qcrbar:\WHCSs\rightarrow\tanXcap$ such that
	\begin{enumerate}[label=\roman*)]
		\item $\qcrbar (v)\in\Cqcvf$ for every $v\in\WHCSs$,
		\item $\Prbar\circ{\qcrbar}(v)=v$ for every $v\in\WHCSs$.
	\end{enumerate}
	Moreover, $\qcrbar$ is linear and satisfies
	\begin{equation}\notag
		\abs{\qcrbar(v)}=\qcr(\abs{v})\quad \capa\text{-a.e.\ for every }v\in\WHCSs,
	\end{equation}
	so that $\qcrbar$ is continuous.
\end{thm}

We often omit to write the $\qcrbar$ operator for simplicity of notation. This should cause no ambiguity thanks to the fact that 
\begin{equation}\label{qcrfactorizes}
	\qcrbar(g v)=\qcr(g)\qcrbar(v) \quad\text{for every }g\in\HSs\cap\Lpi\text{ and } v\in\WHCSs\cap\tanXinf.
\end{equation}
This can be proved easily as the continuity of the map $\qcr$ implies that $\qcr(g) \qcrbar(v)$ as above is quasi-continuous and the injectivity of the map $\Prbar$ restricted the set of quasi-continuous vector fields yields the conclusion.
Again by locality, we have that \eqref{qcrfactorizes} holds even for $g\in\Ss\cap\Lpi$.

\bigskip

The following theorem, which is {\cite[Section 1.3]{bru2019rectifiability}}, is crucial in the construction of modules tailored to particular measures.
\begin{thm}
	\label{finemodule}
	Let $(\XX,\dist,\mass)$ be a metric measure space and let $\mu$ be a Borel measure finite on balls such that $\mu\ll\capa$. Let also $\MM$ be a $\Lpc$-normed $\Lpc$-module. Define the natural (continuous) projection 
	\begin{equation}\notag
		\pi_\mu:\Lpc\rightarrow\Lp^0(\mu).
	\end{equation}
	We define an equivalence relation $\sim_\mu$ on $\MM$ as 
	\begin{equation}\notag
		v\sim_\mu w \text{ if and only if } \abs{v-w}=0 \quad \mu\text{-a.e.}
	\end{equation} 
	Define the quotient module $\MM_{\mu}^0\defeq{\MM}/{\sim_\mu}$ with the natural  (continuous) projection
	\begin{equation}\notag
		\pibar_\mu:\MM\rightarrow\MM_{\mu}^0.
	\end{equation}
	Then $\MM_{\mu}^0$ is a $\Lp^0(\mu)$-normed $\Lp^0(\mu)$-module, with the pointwise norm and product induced by the ones of $\MM$: more precisely, for every $v\in\MM$ and $g\in\Lpc$,
	\begin{equation}\label{defnproj}
		\begin{cases}
			\abs{\pibar_\mu(v)}\defeq\pi_\mu(\abs{v}),\\
			\pi_\mu(g)\pibar_\mu(v)\defeq\pibar_\mu( g v).
		\end{cases}
	\end{equation}
	
	If $p\in[1,\infty]$, we set
	\begin{equation}\notag
		\MM^{p}_{\mu}\defeq\left\{ v\in\MM^0_{\mu}:\abs{v}\in\Lp^p(\mu)\right\},
	\end{equation}
	which is a $\Lp^p(\mu)$-normed $\Lp^\infty(\mu)$-module.
	Moreover, if $\MM$ is a Hilbert module, also $\MM_\mu^0$ and $\MM_\mu^2$ are Hilbert modules. 
\end{thm}

The simple proof of this theorem was not given explicitly in the original article. For this reason we decided to include it here, as this result will play a central role in our note.
\begin{proof}[Proof of Theorem \ref{finemodule}]
	The fact that $\MM^0_{\mu}$ and, if $p\in[1,\infty]$, $\MM_{\mu}^p$ have a well defined structure of normed modules is a trivial verification. Also the statement concerning the Hilbertianity is trivial, as the pointwise parallelogram identity passes to the quotient. The continuity of the projection maps can be readily checked using the characterization of the topologies involved. We still have to verify that the modules $\MM_\mu^p$ are complete for $p\in\{0\}\cup[1,\infty]$. We start by showing that $\MM_\mu^0$ is complete.
	
	Recall that the distance on $\MM_\mu^0$ is defined as $$\dist_0(v,w)\defeq\int_\XX \abs{v-w}\wedge 1\dd{\mu'},$$
	where $\mu'\in\Prob(\XX)$ is such that $\mu\ll\mu'\ll\mu$. We can therefore assume that $\mu$ is finite and use $\mu$ in place of $\mu'$ in the definition of the distance $\dist_0$.
	Take now a Cauchy sequence $\{v_n\}_n\subseteq\MM_\mu^0$, we have to show that it has a convergent subsequence. Notice that we can use Chebyshev's inequality for $\epsilon\in(0,1)$ $$\mu(\{\abs{v-w}>\epsilon\})\le\frac{\dist_0(v,w)}{\epsilon}\quad\text{for every }v,w\in\MM_\mu^0,$$ to extract a (non relabelled) subsequence such that for every $n\in\NN$, $n\ge 1$, it holds
	$$ \mu(\{\abs{v_n-v_{n+1}}> 2^{-n}\})< 2^{-n}.$$
	Take $w_n\in\MM$ such that $v_n=\pibar(w_n)$ for every $n$. We set then $A_n\defeq \{\abs{w_n-w_{n+1}}>2^{-n}\}$ (using the $\capa$ pointwise norm) and $A^k\defeq\bigcup_{n\ge k} A_n$. Notice that $\mu(A^k)\rightarrow 0$ and that, if $n\ge k$, $\abs{w_n-w_{n+1}}\le 2^{-n}\ \capa$-a.e.\ on $\XX\setminus A^k$. We can therefore verify that $\{\chi_{\XX\setminus A^k}w_n\}_n\subseteq\MM$ is a Cauchy sequence, so that it has a limit $w^k$. Then it holds that ${\chi_{\XX\setminus A_k}v_n}=\pibar{(w_n\chi_{\XX\setminus A^k})}\rightarrow\pibar{(w^k)}$ in $\MM_\mu^0$. We can then use the gluing property to define $w\in\MM$ on $\XX\setminus \bigcap_k A^k$ setting $w=w^k$ on $\XX\setminus A^k$. We set $v\defeq\pibar{(w)}$ and we have that $v_n\rightarrow v$ in $\MM_\mu^0$. 
	
	As Chebyshev's inequality implies that Cauchy sequences with respect to the $\MM_\mu^p$ norm are Cauchy sequences with respect to the $\MM_\mu^0$ distance, we obtain the completeness also of $\MM_\mu^p$ for $p\in[1,\infty]$ with standard arguments.
\end{proof}
In the particular case in which $\MM=\tanXcap$ and $\mu$ is a Borel measure finite on balls such that $\mu\ll\capa$, we set
\begin{equation}\notag
	\tanbvXp{p}{\mu}\defeq(\tanXcap)_\mu^p\quad\text{for }p\in\{0\}\cup[1,\infty].
\end{equation} 
In the case $\mu=\mass$ notice that considering the map
$$\dot{\nabla}:\TestF(\XX)\stackrel{\nablatilde}{\longrightarrow}\tanXcap \stackrel{\pibar_\mass }{\longrightarrow}(\tanXcap)_\mass^0 $$we can show that $(\tanXcap)_\mass^0$ is isomorphic to the usual $\Lp^0$ tangent module via a map that sends $\nabla f$ to $\dot{\nabla} f$ so that we have no ambiguity of notation and, by construction, the map $\pibar_\mass$ coincides with $\Prbar$ defined in Proposition \ref{prbardef}.
We define the traces
\begin{alignat}{5}\notag
	&\tr_\mu:\HSsloc\rightarrow\Lp^0( \mu)&&\quad\text{as}\quad &&&\tr_\mu\defeq\pi_{\mu}\circ\qcr,\\\notag
	&\trbar_\mu:\WHCSs\rightarrow\tanbvXzero{\mu}&&\quad\text{as}\quad&&&\trbar_\mu\defeq \pibar_{\mu}\circ\qcrbar.
\end{alignat}

To simplify the notation, we often omit to write the trace operators. This should cause no ambiguity because from \eqref{qcrfactorizes} and \eqref{defnproj} it follows that

\begin{equation}\label{qcrfactorizes2}
	\trbar_\mu(g v)=\tr_\mu (g )\trbar_\mu(v) \quad\text{for every }g\in\HSsloc\cap\Lpi\text{ and } v\in\WHCSs\cap\tanXinf.
\end{equation}

We define
\begin{equation}\notag
	\TestV_\mu(\XX)\defeq\trbar_\mu(\TestV(\XX))\subseteq\tanbvXp{\infty}{\mu}
\end{equation}
and the proof of \cite[Lemma 2.7]{bru2019rectifiability} gives the following result.
\begin{lem}\label{densitytestbargen}
	Let $(\XX,\dist,\mass)$ be an $\RCD(K,\infty)$ space and let $\mu$ be a finite Borel measure such that $\mu\ll\capa$. Then $\TestV_\mu(\XX)$ is dense in $\tanbvXp{p}{\mu}$ for every $p\in[1,\infty)$.
\end{lem}
It is natural to denote
\begin{equation}\notag
	\TestV_F(\XX)\defeq\trbar_{F}(\TestV(\XX))
\end{equation}
and Lemma \ref{densitytestbargen} reads as follows.
\begin{lem}\label{densitytestbar}
	Let $(\XX,\dist,\mass)$ be an $\RCD(K,\infty)$ space and $F\in\BVv^n$. Then $\TestV_F(\XX)$ is dense in $\tanbvX{F}$.
\end{lem}

\bigskip 

We also need Cartesian products of normed modules.
Fix $n\in\NN$, $n\ge 1$ and denote by $\Vert\,\cdot\,\Vert_e$ the Euclidean norm of $\RR^n$. 
Given a $\Lpo$-normed $\Lpo$-module $\mathcal{N} $, we can consider its Cartesian product $\mathcal{N}^n$ and endow it with the natural module structure and with the pointwise norm
$$\abs{(v_1,\dots,v_n)}\defeq\Vert (\abs{v_1},\dots,\abs{v_n})\Vert_e$$
which is induced by a scalar product if and only if the one of $\mathcal{N} $ is, and if this is the case, we still denote the pointwise scalar product on $\mathcal{N}^n$ by $\,\cdot\,$. We endow $\mathcal{N}^n$ with the norm induced by the Lebesgue norm of the relevant exponent of the pointwise norm with respect to $\mass$.
Also, $\mathcal{N}^n$ is a $\Lpt$-normed $\Lpi$-module if and only if $\mathcal{N}$ is, and, if this is the case, a subspace $\mathcal{N}_1$ of $\mathcal{N}$ is dense if and only if $(\mathcal{N}_1)^n$ is dense in $\mathcal{N}^n$. Similar considerations hold if $\mass$ is replaced by a Borel measure, finite on balls and (with the suitable interpretation) in the case of $\Lpc$-normed $\Lpc$-modules or if we alter the integrability exponent. It is clear that if $\MM$ is a $\Lpc$-normed $\Lpc$-module and $\mu$ is a Borel measure finite on balls such that $\mu\ll\capa$, then also $$(\MM_\mu^p)^n=(\MM^n)_\mu^p\quad\text{for }p\in\{0\}\cup[1,\infty].$$
We adopt the natural notation
$$ \mathrm{L}_\mu^p(T^{ n}\XX)\defeq \tanbvXp{p}{\mu}^{ n}$$
and, when possible, we endow  $\mathrm{L}_\mu^p(T^{ n}\XX)$ with the norm induced by the $ \mathrm{L}^p(\mu)$ norm of the (Euclidean) pointwise norm $\abs{\,\cdot\,}$.

The following remark will be used in the sequel without further notice: if $v=(v_1,\dots,v_n)\in\mathcal{N}^n$ is such that for every $i=1,\dots,n$, $v_i\in\HSs$, then $\abs{v}\in\HSs$. This follows from the fact that if $f_1, \dots, f_n\in\HSs$ and $\varphi\in\LIP(\RR^n;\RR)$ is such that $\varphi(0)=0$, then $\varphi(f_1,\dots,f_n)\in\HSs$.

\subsection{The distributional differential}
\label{sectBVRCD}

Let $(\XX,\dist,\mass)$ be an $\RCD(K,\infty)$ space and let $F\in\BVv^n$, where $n\in\NN$, $n\ge 1$ is fixed. Recall that \eqref{diffllcapeq} states that $\abs{\DIFF F}\ll\capa$.
Therefore we can repeat the construction performed in  \cite[Section 2]{bru2019rectifiability} (see also \cite[Subsection 1.3]{bru2019rectifiability}): we employ Theorem \ref{finemodule} to obtain the couple of modules  $\tanbvXzero{F}\defeq\tanbvXzero{\abs{\DIFF F}}$ and $\tanbvX{F}\defeq\tanbvX{\abs{\DIFF F}}$ as well as the couple of traces $\tr_F\defeq\tr_{\abs{\DIFF F}}$ and $\trbar_F\defeq\trbar_{\abs{\DIFF F}}$, which we will often omit to write.

\bigskip 

In view of the following theorem, recall that the interpretation of the integral in \eqref{difffloccaompeq2} is given by Remark \ref{interpretationint}. Recall also \eqref{Cqcvfdef}.

\begin{thm}\label{weakder2}
	Let $(\XX,\dist,\mass)$ be an $\RCD(K,\infty)$ space and $F\in\BVv^n$.
Then there exists a unique vector field $\nu_F\in\mathrm{L}_F(T^{ n}\XX)$ such that it holds 
	\begin{equation}\label{difffloccaompeq2}
		\sum_{i=1}^n\int_\XX F_i\dive\, v_i=-\int_\XX v\,\cdot\,\nu_F\dd{\abs{\DIFF F}}\quad\text{for every }v=(v_1,\dots,v_n)\in (\Cqcvfinf\cap D(\dive))^n.
	\end{equation}
	Moreover, $\abs{\nu_F}=1\ \abs{\DIFF F}$-a.e.
\end{thm}

\begin{proof} 
	We divide the proof in several steps.
	\\\textsc{Step 1}. We show that if $f\in\BVv\cap\Lpi$, then there exists a unique $\nu_f\in \mathrm{L}_f(T\XX)$ such that 
		\begin{equation}\label{difffloccaompeq2scarsa}
		\int_\XX f\dive\, v=-\int_\XX v\,\cdot\,\nu_f\dd{\abs{\DIFF f}}\quad\text{for every }v\in\WHCSs\cap D(\dive)\cap\tanXinf,
	\end{equation}
and moreover $\abs{\nu_f}=1\ \abs{\DIFF f}$-a.e.
	This can be proved following \textit{verbatim} the proof of \cite[Theorem 2.2]{bru2019rectifiability}. Notice that  in \cite{bru2019rectifiability} the assumption that the dimension was finite could have been dropped taking into account Theorem \ref{diffllcapthm} and Proposition \ref{bakryemry}.
	\\\textsc{Step 2}. We show that for $f\in\BVv\cap\Lpi$, \eqref{difffloccaompeq2scarsa} holds for every $v\in\Cqcvfinf\cap D(\dive)$. 
		Fix then $v\in \Cqcvfinf\cap D(\dive)$. By an easy cut-off argument, there is no loss of generality in assuming that $v$ has bounded support.  We take a sequence $\{t_k\}_k\subseteq(0,1)$ with $t_k\searrow 0$, then, we define $\{v_k\}_k$ as $v_k\defeq \psi e^{K t_k}\heat_{\mathrm{H},t_k} v$, where $\psi\in\LIPbs(\XX)$ (say $\supp\psi\Subset B$ for some ball $B$) is a cut-off function that is identically $1$ on a neighbourhood of $\supp v$. Now notice that the equality in \eqref{difffloccaompeq2scarsa} holds for $v_k$ and that, as $\dive\, v_k\rightarrow\dive\, v$ in $\Lpt$, it suffices to check that $v_k\rightarrow v$ in $\tanXcap$, then dominated convergence together with the theory developed in \cite{debin2019quasicontinuous} implies the conclusion.
	
	We thus have to show that $v_k\rightarrow v$ in in $\tanXcap$.  We can assume with no loss of generality that $\abs{v}\le 1\ \mass$-a.e.\ so that also $|v_k|\le 1\ \mass$-a.e. for every $k$. As $v\in\Cqcvf$, we can take a sequence $\{w_l\}_l\subseteq\TestV(\XX)$ such that $w_l\rightarrow v$ in $\tanXcap$, $\supp w_l\Subset B$ and $\abs{w_l}\le 1\ \mass$-a.e.  We claim that 
	\begin{equation}\label{unifk}
		w_{l,k}\defeq \psi e^{K t_k}\heat_{\mathrm{H},t_k} w_l\rightarrow \psi e^{K t_k}\heat_{\mathrm{H},t_k} v= v_k\text{ uniformly in $k$} \quad\text{in } \tanXcap \text{ as $l\rightarrow\infty$}.
	\end{equation} 
	Fix for the moment $\epsilon>0$. As $w_l\rightarrow v$ in $\tanXcap$ and the fact that all these vector fields have uniformly bounded support, we can take functions $\{g_l\}_l\subseteq\HSs$ such that $\Vert g_l\Vert_{\HSs} \rightarrow 0$, $g_l(x)\in[0,1]$ for $\mass$-a.e.\ $x$ and
	$$\left\{|{w_l-v}|>\epsilon\right\}\subseteq\left\{g_l\ge 1\right\} .$$
	Therefore, taking into account that 
	\begin{equation}\notag
		\begin{split}
			|w_{l,k}-v_k|&\le \chi_B\left( e^{K t_k} |\heat_{\mathrm{H},t_k}\left((w_l-v)\chi_{\{|w_l-v|>\epsilon\}}\right)|+e^{K t_k}|\heat_{\mathrm{H},t_k}\left((w_l-v)\chi_{\{|w_l-v|\le\epsilon\}}\right)|\right) \\&\le\chi_B\left( \heat_{t_k}(|(w_l-v)\chi_{\{|w_l-v|>\epsilon\}}|)+\heat_{t_k}(|(w_l-v)\chi_{\{|w_l-v|\le\epsilon\}}|)\right)\le 2\chi_B \heat_{t_k} g_l+2 \epsilon
		\end{split}
	\end{equation}
	and that 
	$$  \Vert \heat_{t_k} g_l\Vert_{\HSs}\le \Vert g_l\Vert_{\HSs}\rightarrow 0\text{ uniformly in $k$}\quad\text{as }l\rightarrow\infty,$$
	it follows that, uniformly in $k$,
	$$\limsup_l \capa\left\{|w_{l,k}-v^k|> 4\epsilon \right\} \le \limsup_l \capa\left(B\cap \left\{|\heat_{t_k} g_l|> \epsilon \right\}\right)=0$$
	and hence \eqref{unifk} follows. Now we can conclude easily, noticing that $$\dist_{\tanXcap}(v_k,v)\le \dist_{\tanXcap}(v_k,w_{l,k})+\dist_{\tanXcap}(w_{l,k},w_l)+\dist_{\tanXcap}(w_{l},v) $$
	as we can first take $l$ large enough  to estimate the first and last summand (uniformly in $k$)
	and then let $k\rightarrow \infty$, recalling that as $w_{l,k}\rightarrow w_l$ in $\WHHSs$,  $w_{l,k}\rightarrow w_l$ in $\tanXcap$.
		\\\textsc{Step 3}. We drop the $\Lpi$ bound assumption on $f$ made in Step 1. For $k\in\mathbb{Z}$, define $$f_k\defeq (f\vee k)\wedge (k+1)-(k+\chi_{\{k< 0\}}(k)).$$
Notice that for every $k\in\mathbb{Z}$, $\abs{\DIFF f_k}\le\abs{\DIFF f}$ and, as $f_k\in\BVv\cap\Lpi$, there exists $\nu_{f_k}\in\tanbvX{f_k}$ such that $\abs{\nu_{f_k}}=1\ \abs{\DIFF f_k}$-a.e.\ and
\begin{equation}\label{ggscarsa}
	\int_\XX f_k  \dive\, v\dd{\mass}=-\int_\XX v\,\cdot\,\nu_{f_k} \dd{\abs{\DIFF f_k}}
\end{equation}
for every $v\in \Cqcvfinf\cap D(\dive)$.

Notice that, if we consider ${\nu}_{f_k}$ as an element of $\tanXcap$, we have that $\pibar_{f}( {\nu}_{f_k})\dv{\abs{\DIFF f_k}}{\abs{\DIFF f}}$ is well defined.
Now notice that by \eqref{coareaeqdiff} it holds $\abs{\DIFF f}=\sum_{k\in\mathbb{Z}}\abs{\DIFF f_k}$.
Then, as $$\bigg\Vert\pibar_{f}( \tilde{\nu}_{f_k})\dv{\abs{\DIFF f_k}}{\abs{\DIFF f}}\bigg\Vert_{\tanbvX{f}} \le \abs{\DIFF f_k}(\XX),$$  and the completeness of $\tanbvX{f}$, we see that $$\nu_f\defeq\sum_{k\in\mathbb{Z}}\pibar_{f}( \tilde{\nu}_{f_k})\dv{\abs{\DIFF f_k}}{\abs{\DIFF f}}$$ 
is a well defined element of $\tanbvX{f}$ and satisfies $\abs{\nu_f}\le 1\ \abs{\DIFF f}$-a.e. 

Fix now $v\in \Cqcvfinf\cap D(\dive)$.
We prove the integration by parts formula \eqref{difffloccaompeq2scarsa} for $f$ and $v$. Using \eqref{ggscarsa} we have,
\begin{equation}\notag
	\begin{split}
		&\int_\XX f  \dive\, v\dd{\mass}=\lim_m\int_\XX \sum_{k=-m}^{m-1} f_k  \dive\, v\dd{\mass}\\&\quad=-\lim_m\int_\XX v\,\cdot\,\sum_{k=-m}^{m-1}\pibar_{f}(\tilde{\nu}_{f_k}) \dv{\abs{\DIFF f_k}}{\abs{\DIFF f}}\dd{\abs{\DIFF f}}=-\int_\XX v\,\cdot\,\nu_f\dd{\abs{\DIFF f}},
	\end{split}
\end{equation} 
where the above computation also shows the existence of the limit $$\lim_m\int_\XX (f\vee -m)\wedge m\, \dive\,v \dd{\mass}=\lim_m\int_\XX \sum_{k=-m}^{m-1} f_k  \dive\, v\dd{\mass}.$$
Now we can show that $\abs{\nu_f}\ge 1\ \abs{\DIFF f}$-a.e.\ arguing as in the proof of \cite[Theorem 2.2]{bru2019rectifiability}. Finally, uniqueness of $\nu_f$ follows from Lemma \ref{densitytestbar}. All in all, we have proved the theorem in the case $n=1$.
\\\textsc{Step 4}.
We prove the theorem for $n>1$. Notice that $\abs{\DIFF F}$ is a finite measure such that 
\begin{equation}\notag
\abs{\DIFF F}\le \abs{\DIFF F_1}+\dots+\abs{\DIFF F_n}\ll \capa.
\end{equation}
We can therefore define $\tanbvXzero{F}\defeq\tanbvXzero{\abs{\DIFF F}}$ and $\tanbvX{F}\defeq\tanbvX{\abs{\DIFF F}}$ as well as the traces $\tr_F\defeq\tr_{\abs{\DIFF F}}$ and $\trbar_F\defeq\trbar_{\abs{\DIFF F}}$ (that, as usual, will not be written). 
Also, taking into account that $\abs{\DIFF F_i}\le \abs{\DIFF F}$, we can write, thanks to the Radon-Nikodym Theorem,
\begin{equation}\notag
\abs{\DIFF F_i}=g_i\abs{\DIFF F}\quad\text{for every }i=1,\dots,n 
\end{equation}
where $g_i\in\Lp^\infty(\XX,\abs{\DIFF F})$.
This discussion and the construction done exploiting Theorem \ref{finemodule} imply that for every $i=1,\dots,n$ we can set
\begin{equation}\notag
\nu_i\defeq g_i\pibar_{F}({{\nu}_{F_i}})\in\tanbvX{F},
\end{equation} 
where we are considering ${\nu}_{F_i}$ as a element of $\tanXcap$. Notice that as ${\nu}_{F_i}$ is well defined $\abs{\DIFF F_i }$-a.e. $\nu_i$ is well defined and that $\abs{\nu_i}=\abs{g_i}\ \abs{\DIFF F}$-a.e. 

We set $\nu_F\defeq(\nu_1,\dots,\nu_n)\in\mathrm{L}^2_F(T^{ n}\XX)$ and therefore \eqref{difffloccaompeq2} holds. As uniqueness of such $\nu_F$ follows from the same fact in the unidimensional case, to conclude it remains only to show that 
\begin{equation}\notag
	\abs{\nu_F}=1\quad{\abs{\DIFF F}}\text{-a.e.}
\end{equation}

By density, take a sequence $\{w^k=(w_1^k,\dots,w_n^k)\}\subseteq\TestV(\XX)^n$ such that $w_k\rightarrow \chi_{\{\abs{\nu_F}> 0\}}\frac{\nu_F}{\abs{\nu_F}}$ in $\mathrm{L}^2_F(T^{ n}\XX)$. Then define $\{v^k\}_k\subseteq \TestV(\XX)^n$ as $$v^k\defeq \frac{1}{1\vee\abs{w^k}}w^k.$$ Notice that still $v^k\rightarrow \chi_{\{\abs{\nu_F}> 0\}}\frac{\nu_F}{\abs{\nu_F}}$ and moreover $|{v^k}|\le 1\ \mass$-a.e. 
Let $A\subseteq \XX$ open and
take a sequence $\{\psi_k\}_k\subseteq\LIPbs(\XX)$ such that $\psi_k(x)\in[0,1]$ for every $x\in\XX$, $\supp\psi_k\subseteq A$ and $\psi_k(x)\nearrow 1$ for every $x\in A$. 
By Proposition \ref{reprvett1} and \eqref{difffloccaompeq2} we can compute  
$$\abs{\DIFF F}(A)\ge - \sum_{i=1}^n \int_\XX F_i\dive(\psi_k v_i^k)\dd{\mass}=\int_\XX \psi_k v^k\,\cdot\,\nu_F\dd{\abs{\DIFF F}}\rightarrow \int_A \abs{\nu_F}\dd{\abs{\DIFF F}},$$
and this shows that $\abs{\nu_F}\le 1\ \abs{\DIFF F}$-a.e. 	
	Now notice that Proposition \ref{reprvett1} and \eqref{difffloccaompeq2} again imply  that
	$$\int_\XX\abs{\nu_F}\dd{\abs{\DIFF F}}\ge 1$$ so that we conclude.
\end{proof}

\begin{rem}\label{normcomp}
	Let $F=(F_1,\dots,F_n)\in\BVv^n$. Then,
	\begin{equation}\label{coordnorm}
		(\nu_F)_i =\dv{\abs{\DIFF F_i}}{\abs{\DIFF F}}\nu_{F_i} \quad\abs{\DIFF F}\text{-a.e.\ for every }i=1,\dots,n,
	\end{equation}
which is an immediate consequence of \eqref{difffloccaompeq2}.
\fr
\end{rem}

\begin{rem}
	We show that if $F\in\BVv^n$, then there exists a sequence
				 $$\{v^k=(v_1^k,\dots,v_n^k)\}_k\subseteq\mathcal{W}^n$$where$$\mathcal{W}^n\defeq\Big\{v=(v_1,\dots v_n)\in\WHHSsn:\abs{v}\le 1\ \mass\text{-a.e.\ } \dive\,v_i\in\Lpi\text{ for every }i=1,\dots,n\Big\}$$
such that $v^k\rightarrow \nu_F$ in $\tanbvXn{F}$.
	
	Indeed, we can modify the proof of Proposition \ref{reprvett1} (see in particular its Step 3), replacing $\TestV(\XX)^n$ with $\mathcal{W}^n$ in \eqref{mathvvet}  and then it is enough to take a sequence $\{v^k=(v^k_1,\dots,v_n^k)\}_k\subseteq\mathcal{W}^n$  such that (with the usual interpretation for the integral) $$\sum_{i=1}^n\int_\XX F_i \dive\, v^k_i\dd{\mass}\rightarrow-\abs{\DIFF F}(\XX)$$
	and compute
	\begin{equation*}\notag
		\begin{split}
			&\int_\XX |{v^k-\nu_F}|^2\dd{\abs{\DIFF F}}=\int_\XX |{v^k}|^2\dd{\abs{\DIFF F}}+\int_\XX \abs{\nu_F}^2\dd{\abs{\DIFF F}}-2 \int_\XX v^k\,\cdot\,\nu_F\dd{\abs{\DIFF F}}\\&\quad\le 2\abs{\DIFF F}(\XX)+2 \sum_{i=1}^n\int_\XX F_i \dive\, v^k_i\dd{\mass}\rightarrow 0,
		\end{split}
	\end{equation*}
	where in the last inequality we used \eqref{difffloccaompeq2}. Notice that this argument works only to approximate $\nu_F$, and the difficulty in approximating other elements of $\tanbvXn{F}$ lies in the request of essentially bounded divergence request.
	\fr
\end{rem}

\bigskip

We introduce now a formal framework to denote the distributional differential of functions of bounded variation and to be able to perform algebraic manipulations on this object. In a forthcoming work, we will introduce the concept of module valued measure, which is a suitable notion to describe the distributional differential and its properties. 
\begin{defn}\label{mvmeas}
Let $(\XX,\dist,\mass)$ be an $\RCD(K,\infty)$ space. 	Let $\nu\in\tanXcapinfn$ and let $\mu$ be a finite measure with $\mu\ll\capa$. We write $\nu\mu$ to denote the object that acts on $\tanXcapinfn$ as follows:
\begin{equation}\notag
	\nu\mu(\XX)(v)\defeq \int_\XX v\,\cdot\,\nu\dd{\mu}\quad\text{for every }v\in \tanXcapinfn.
\end{equation} 
Given $\nu_1\mu_1$ and $\nu_2\mu_2$, we write $\nu_1\mu_1=\nu_2\mu_2$ if and only if 
\begin{equation}\label{eqmvmeas}
	\nu_1\mu_2(\XX)(v)=\nu_1\mu_2(\XX)(v)\quad\text{for every }v\in \tanXcapinfn.
\end{equation} 
\end{defn}
\begin{rem}
Given $\mu$ and $\nu$ as above, we can of course define \begin{equation}\label{defmunu}
	\nu\mu(A)(v)\defeq \int_A v\,\cdot\,\nu\dd{\mu}\quad\text{for every }v\in \tanXcapinfn\text{ and every $A\subseteq\XX$ Borel}.
\end{equation} 
More generally, objects taking as entries subsets of $\XX$ and vector fields and that satisfy additional properties (these properties are, for example, satisfied by $\nu\mu$ defined according to \eqref{defmunu}) will be called module valued measures and will be the subject of a forthcoming work of the authors.
\fr
\end{rem}
\begin{rem}
Notice that the expression $\nu\mu$ makes sense even if we only know that the vector field $\nu$ is defined $\mu$-a.e. We will exploit this fact throughout.

Also, we have that \eqref{eqmvmeas} holds if and only if 
\begin{equation}\notag
	\nu_1\mu_1(\XX)(v)=\nu_2\mu_2(\XX)(v)\quad\text{for every }v\in \TestV(\XX)^n.
\end{equation} 
Indeed, we can take $\mu\defeq\mu_1+\mu_2$ and notice that $\nu_i\mu_i=\left( \nu_i\dv{\mu_i}{\mu} \right)\mu$ for $i=1,2$, then the claim follows by density. \fr
\end{rem}

We define now some formal algebraic operations for objects of the kind $\nu\mu$.
\begin{defn}
		Let $(\XX,\dist,\mass)$ be an $\RCD(K,\infty)$ space. 	
\begin{enumerate}[label=\roman*)]\label{formalcalc}
\item Let $\nu\in\tanXcapinfn$ and let $\mu$ be a finite measure with $\mu\ll\capa$. Let moreover $\phi:\XX\rightarrow\RR^{m\times n}$ be a $\mu$-measurable function. We define $\phi( \nu\mu)\defeq(\phi\nu)\mu$, where \begin{equation}\notag
	(\phi\nu)_j\defeq\sum_{i=1}^n \phi_{j,i}\nu_i\quad j=1,\dots,m.
\end{equation}
\item	Let $\nu_1,\nu_2\in\tanXcapinfn$ and let $\mu_1,\mu_2$ be two finite measures with $\mu_1\ll\capa,\mu_2\ll\capa$. We define $\nu_1\mu_1+\nu_2\mu_2$ as follows. First, define $\mu\defeq\mu_1+\mu_2$, then set 
	\begin{equation}\notag
		\nu_1\mu_1+\nu_2\mu_2=\left(\nu_1 \dv{\mu_1}{\mu}+\nu_2 \dv{\mu_2}{\mu}\right)\mu.
	\end{equation}
\end{enumerate}
\end{defn}

\begin{rem}
	We keep the same notation as in the definition above.
\begin{enumerate}[label=\roman*)]
	\item For what concerns item $i)$,	\begin{equation}\notag
		(\phi\nu\mu)(\XX)(v)=\int_\XX \sum_{j=1}^m\sum_{i=1}^n v_j\,\cdot\, \phi_{j,i}\nu_i\dd{\mu}
	\end{equation}
	for every $v\in\tanXcapinfm$.
\item	For what concerns item $ii)$,
	\begin{equation}\notag
		(\nu_1\mu_1+\nu_2\mu_2)(\XX)(v)=\int_\XX v\,\cdot\,\nu_1\dd{\mu_1}+\int_\XX v\,\cdot\,\nu_2\dd{\mu_2}
	\end{equation}
	for every $v\in\tanXcapinfn$.\fr
\end{enumerate}
\end{rem}
\begin{defn}
	Let $(\XX,\dist,\mass)$ be an $\RCD(K,\infty)$ space and $F\in\BVv^n$. We define $\DIFF F\defeq \nu_F\abs{\DIFF F}$, according to Definition \ref{mvmeas}.
\end{defn}
\begin{rem}
	Let $F\in\BVv^n$. Then \eqref{difffloccaompeq2} reads
	\begin{equation}\notag
		\sum_{i=1}^n \int_\XX F_i\dive v_i=-\DIFF F(\XX)(v)\quad\text{for every }v=(v_1,\dots,v_n)\in (\Cqcvfinf\cap D(\dive))^n. 
	\end{equation}
	Also, if $\mu\nu$ is as in Definition \ref{mvmeas} and satisfies  		\begin{equation}\notag
		\sum_{i=1}^n \int_\XX F_i\dive v_i=-\mu\nu (\XX)(v)\quad\text{for every }v=(v_1,\dots,v_n)\in \TestV(\XX)^n,
	\end{equation}
	then $\DIFF F=\nu\mu$.
	
	Finally, it is clear that the map $\BVv^n\ni F\mapsto\DIFF F$ is linear.\fr
\end{rem}

\subsection{Fine properties} Before proving the calculus rules satisfied by the distributional differential of functions of bounded variation, we need some preliminary results regarding the fine properties of functions of bounded variation on $\RCD(K,N)$ spaces. 

\begin{lem}\label{convergenceprec}
	Let $(\XX,\dist,\mass)$ be an $\RCD(K,N)$ space and $f\in\BVv\cap\Lpi$. Then 
	\begin{equation}\notag
		\lim_{s\searrow 0}\heat_s f(x)= \bar{f}(x)\quad \HH^h\text{-a.e.}
	\end{equation}
\end{lem}
\begin{proof}
	In the sequel, let $C$ denote a numerical constant depending only on the parameters entering into play (it may vary during the proof).
	\\\textsc{Step 1}. Case $x\in\XX\setminus S_f$. We can compute 
	\begin{equation}\notag
		\limsup_{s\searrow 0}\abs{\heat_{s^2} f(x)-\bar{f}(x)}\le \limsup_{s\searrow 0}\int_\XX p_{s^2}(x,y)\abs{f(y)-\bar{f}(x)}\dd{\mass(y)}.
	\end{equation}
	Fix now $a>1$. Using the estimates for the heat kernel in \eqref{heatkerneleq}
	\begin{equation}\notag
		\int_{B_{a s}(x)} p_{s^2}(x,y)\abs{f(y)-\bar{f}(x)}\dd{\mass(y)}\le \frac{C}{\mass (B_s(x))} \int_{B_{a s}(x)}\abs{f(y)-\bar{f}(x)}\dd{\mass(y)}
	\end{equation}
	that converges to $0$, because of the doubling inequality and the fact that points in $\XX\setminus S_f$ are Lebesgue points for $f$. Also, as $f\in\Lpi$,
	\begin{equation}\notag
		\int_{\XX\setminus B_{a s}(x)} p_{s^2}(x,y)\abs{f(y)-\bar{f}(x)}\dd{\mass(y)}\le C \int_{\XX\setminus B_{a s}(x)} p_{s^2}(x,y)\dd{\mass(y)}.
	\end{equation}
	Using again the estimates for the heat kernel in \eqref{heatkerneleq},
	\begin{equation}\notag
		\begin{split}
			p_{s^2}(x,y)&\le \frac{ C}{\mass(B_s(x))} \exp{-\frac{\dist(x,y)^2} {6 s^2}}\\&\le C H(x,y) \frac{1}{\mass(B_{2 s}(x))}\exp{-\frac{\dist(x,y)^2} {8 s^2}} \le C H(x,y) p_{4 s^2}(x,y),
		\end{split}
	\end{equation}
	where, using the doubling inequality,
	$$ H(x,y)\defeq\frac{\mass(B_{2 s}(x))}{\mass(B_s(x))} \exp{-\frac{\dist(x,y)^2}{6 s^2}+\frac{\dist(x,y)^2} {8 s^2} } \le C e^{-a^2/24}\quad\text{if\ }\dist(x,y)\ge a s.$$ Therefore
	\begin{equation}\label{tmphhhh}
		\int_{\XX\setminus B_{a s}(x)} p_{s^2}(x,y)\dd{\mass(y)}\le C e^{-a^2/24} \int_\XX p_{4 s^2}(x,y)\dd{\mass(y)}\le C e^{-a^2/24}.
	\end{equation}
	Being $a$ arbitrary, we conclude $\lim_{s\searrow 0}\heat_{s^2} f(x)=\bar{f}(x)$ if $x\notin S_f$.
	\\\textsc{Step 2}. Case $x\in S_f$. Let $D\subseteq\RR$ a countable dense set such that if $t\in D$, then $E_t\defeq\{f>t\}$ is a set of finite perimeter. This is possible thanks to the coarea formula \eqref{coareaeqdiff}. Set, for every $t\in D$, $N_t$ as the set of points of $\partial^* E_t$ where the conclusions of Theorem \ref{uniquenesstaper} or Lemma \ref{compactnessfp} fail.
	We know that $\abs{\DIFF\chi_{E_t}}(N_t)=0$, hence $\HH^h(N_t)=0$, \eqref{reprformula3}. We set $$N\defeq \bigcup_{t \in D}N_t,$$ notice $\HH^h(N)=0$. We show now $\lim_{s\searrow 0}\heat_{s^2} f(x)=\bar{f}(x)$ for every $x\in S_f\setminus N$. Fix $x\in S_f\setminus N$ and $t\in (f^{\wedge}(x),f^{\vee}(x))\cap D$. By the very definition, $x\in\partial^* E_t$, therefore at $x$ the conclusions of Theorem \ref{uniquenesstaper} and Lemma \ref{compactnessfp} hold. As a consequence, using also \cite[(4.13)]{ambrosio2018rigidity}, we infer (if we let $n$ denote the essential dimension of the space)
	$$\lim_{s\searrow 0}\heat_{s^2}(\chi_{E_t})(x)=\heat^{\RR^n}_1 (\chi_{\{x_n>0\}})(0)=1/2,$$ so that also $$\lim_{s\searrow 0}\heat_{s^2}(\chi_{\XX\setminus E_t})(x)=1/2.$$
	We can then write
	\begin{equation}\notag
		\begin{split}
			&\limsup_{s\searrow 0}\abs{\heat_{s^2} f(x)-\bar{f}(x)}\\&\quad \le\limsup_{s\searrow 0} \abs{\heat_{s^2}(\chi_{\XX\setminus E_t} (f -f^{\wedge}(x)))}(x)+ \limsup_{s\searrow 0} \abs{\heat_{s^2}(\chi_{E_t} (f -f^{\vee}(x))) }(x).
		\end{split}
	\end{equation}
	It is enough then to show $\limsup_{s\searrow 0} \abs{\heat_{s^2}(\chi_{E_t} (f -f^{\vee}(x))) }(x)=0$ (the other term can be dealt similarly).
	We fix $a>1$. We can compute
	\begin{equation}\notag
		\begin{split}
			\abs{\heat_{s^2}(\chi_{E_t} (f -f^{\vee}(x))) }(x)&\le \int_{E_t} p_{s^2}(x,y) \abs{f(y)-f^{\vee}(x)}\dd{\mass(y)}\\&= \int_{E_t\cap B_{a s}(x)} p_{s^2}(x,y) \abs{f(y)-f^{\vee}(x)}\dd{\mass(y)}\\&\qquad+ \int_{\XX\setminus B_{a s}(x)} p_{s^2}(x,y) \abs{f(y)-f^{\vee}(x)}\dd{\mass(y)}.
		\end{split}
	\end{equation}
	The second term on the right hand side is bounded by $C e^{-a^2/24}$, thanks to $f\in\Lpi$ and \eqref{tmphhhh}. As $a$ is arbitrary, we conclude if we show that the first term converges to $0$ as $s\searrow 0$.
	Using the estimates for the heat kernel in \eqref{heatkerneleq} and the doubling inequality we estimate the first term by
	\begin{equation}\notag
		\begin{split}
			&\limsup_{s\searrow 0}\int_{E_t\cap B_{a s}(x)} p_{s^2}(x,y) \abs{f(y)-f^{\vee}(x)}\dd{\mass(y)}\\&\qquad\le\limsup_{s\searrow 0} \frac{C} {\mass(B_{a s}(x))} \int_{E_t\cap B_{a s}(x)} \abs{f(y)-f^{\vee}(x)}\dd{\mass(y)}.
		\end{split}
	\end{equation}
	
	Take now any $t_1\in (t,f^{\vee}(x))\cap D$ and $t_2\in (f^{\vee}(x),\infty)\cap D$.
	We can split 
	\begin{equation}\notag
		B_{a s}(x)\cap E_t= \left(B_{a s}(x)\cap (E_t\setminus E_{t_1})\right)\cup\left( B_{a s}(x)\cap (E_{t_1}\setminus E_{t_2})\right)\cup\left( B_{a s}(x)\cap E_{t_2}\right).
	\end{equation}
	Now by the very definition of $f^{\vee}(x)$, $E_{t_2}$ has density $0$ at $x$. Also, $E_t\setminus E_{t_1}$ has density $0$ at $x$, as $E_{t_1}\subseteq E_t$ and both $E_t$ and $E_{t_1}$ have density $1/2$ at $x$, as a consequence of $x\in\partial^* E_{t_1}\cap\partial^* E_t$ and $x\notin N $. Therefore, taking into account $f\in\Lpi$, we are left with 
	\begin{equation}\notag
		\begin{split}
			&\limsup_{s\searrow 0} \frac{C} {\mass(B_{a s}(x))} \int_{E_t\cap B_{a s}(x)} \abs{f(y)-f^{\vee}(x)}\dd{\mass(y)}\\&\quad =\limsup_{s\searrow 0} \frac{C} {\mass(B_{a s}(x))} \int_{(E_{t_1}\setminus E_{t_2})\cap B_{a s}(x)} \abs{f(y)-f^{\vee}(x)}\dd{\mass(y)} \le C (t_2-t_1).
		\end{split}
	\end{equation}
	We conclude as we can take $t_2,t_1\rightarrow t$ by density of $D$ in $\RR$.
\end{proof}

\begin{lem}\label{normcut}
Let $(\XX,\dist,\mass)$ be an $\RCD(K,N)$ space and $f,f_1,f_2\in\BVv$ such that $f=f_1+f_2$ and $\abs{\DIFF f}=\abs{\DIFF f_1}+\abs{\DIFF f_2}$. Then 
\begin{alignat}{2}\notag
{\nu}_f&={\nu}_{f_i}&&\abs{\DIFF f_i}\text{-a.e.\ for $i=1,2,$}\\\notag
{\nu}_{f_1}&={\nu}_{f_2}\quad&&\abs{\DIFF f_1}\wedge\abs{\DIFF f_2}\text{-a.e.} 
\end{alignat}
\end{lem}
\begin{proof}
	By Lemma \ref{densitytestbar}, we have a sequence $\{v_k\}_k\subseteq \TestV(\XX)$ such that $v_k\rightarrow\nu_{f}$ in $\tanbvX{f}$. In particular, 
\begin{equation}\notag
	\begin{split}
		\abs{\DIFF f}(\XX)=\lim_k \int_\XX v_k\,\cdot\,\nu_{f} \dd{\abs{\DIFF f}}.
	\end{split}
\end{equation} 
Thanks to the hypothesis $\abs{\DIFF f}=\abs{\DIFF f_1}+\abs{\DIFF f_2}$,
\begin{equation}\notag
	\begin{split}
		&\Vert v_k-\nu_{f_1}\Vert^2_{\tanbvX{f_1}}+\Vert v_k-\nu_{f_2}\Vert^2_{\tanbvX{f_2}}\\&\quad=\Vert \nu_{f_1}\Vert^2_{\tanbvX{f_1}}+\Vert \nu_{f_2}\Vert^2_{\tanbvX{f_2}} +\Vert v_k\Vert^2_{\tanbvX{f}} -2\int_\XX v_k\,\cdot\,\nu_{f_{1}} \dd{\abs{\DIFF f_{1}}}-2\int_\XX v_k\,\cdot\,\nu_{f_2} \dd{\abs{\DIFF f_2}}\\&\quad= \abs{\DIFF f}(\XX)+\Vert v_k\Vert^2_{\tanbvX{f}} -2\int_\XX v_k\,\cdot\,\nu_{f} \dd{\abs{\DIFF f}}
	\end{split}
\end{equation}
where in the last equality we used also the linearity of the map $f\mapsto \DIFF f$.
It follows 
\begin{equation}\notag
	\lim_k \left(\Vert v_k-\nu_{f_1}\Vert^2_{\tanbvX{f_1}}+\Vert v_k-\nu_{f_2}\Vert^2_{\tanbvX{f_2}} \right)=0.
\end{equation}
We conclude as we have proved
$v_k\rightarrow \nu_f$ in $\tanbvX{f}$ and $v_k\rightarrow \nu_{f_i}$ in $\tanbvX{f_i}$ for $i=1,2$.
\end{proof}
\begin{lem}\label{oknormal}
	Let $(\XX,\dist,\mass)$ be an $\RCD(K,N)$ space and $E, F$ two subsets of finite perimeter and finite mass of $\XX$.
	Then
	\begin{equation}\label{lemloc}
		{\nu}_{E}=\pm{\nu}_{F}\quad {\HH^h}\text{-a.e.\ in }\partial^*{E}\cap\partial^*{F}.
	\end{equation} 
	More precisely, there exist two sets $N^+,N^-\subseteq\partial^*E\cap\partial^*F$ with $\HH^h((\partial^*E\cap\partial^*F)\setminus(N^+\cup N^-))=0$ such that the following holds: for every $x\in N^+$ (resp.\ $N^-$) the set $E\Delta F$ has density $0$ (resp.\ $1$) at $x$ and \eqref{lemloc} holds $\HH^h$-a.e.\ in $ N^+$ (resp.\ $N^-$) with the $+$ (resp.\ $-$) sign.
\end{lem}
Notice that \eqref{reprformula3} implies that \eqref{lemloc} is well defined.
\begin{proof}
	First assume $E\subseteq F$.
	In this case \eqref{coareaeqdiff} shows that we can use Lemma \ref{normcut} with $f_1=\chi_{E}$ and $f_2=\chi_{F}$, so that, recalling also \eqref{reprformula3}, $${\nu}_{E}={\nu}_{F}\quad {\HH^h}\text{-a.e.\ in }\partial^*{E}\cap\partial^*{F}.$$ 
	Assume now $E\cap F=\emptyset$. Using the same arguments as above (with $f_1=\chi_{E}$ and $f_2=-\chi_{F}$), $${\nu}_{E}=-{\nu}_{F}\quad {\HH^h}\text{-a.e.\ in }\partial^*{E}\cap\partial^*{F},$$
	where we also used $\nu_{-\chi_F}=-\nu_F$.
	
	Thanks to \eqref{unmezzoqoeq}, we may consider only the set of points at which the sets $E$ and $F$ have density $1/2$ and the sets $E\setminus F$, $F\setminus E$ and $E\Delta F$ have density in $\{0,1/2,1\}$. We easily show that $E\Delta F$ can not have density $1/2$ at such points.
	If $E\Delta F$ has density $0$ at $x$, then $E\cap F$ has density $1/2$ at $x$. We can use the first case treated above to compare first ${\nu}_{E}$ with ${\nu}_{E\cap F}$ and then ${\nu}_{F}$ with with ${\nu}_{E\cap F}$.
	If instead $E\Delta F$ has density $1$ at $x$, both $E\setminus F$ and $F\setminus E$ have density $1/2$ at $x$. We can use the first case treated above to compare first ${\nu}_{E}$ with ${\nu}_{E\setminus F}$, then ${\nu}_{F}$ with ${\nu}_{F\setminus E}$ and conclude comparing ${\nu}_{E\setminus F}$ with ${\nu}_{F\setminus E}$, using the second case treated above.
\end{proof}
\begin{rem}
It is worth noticing that one can easily derive from Lemma \ref{oknormal} the calculus rules for the distributional derivatives of intersection, union and difference of sets of finite perimeter and finite mass, cf.\ \cite[Theorem 4.11]{bru2021constancy}.\fr
\end{rem}

\begin{lem}\label{normsublevel}
	Let $(\XX,\dist,\mass)$ be an $\RCD(K,N)$ space and $f\in\BVv$. Define
	\begin{equation}\notag
		\nu_f^t\defeq
		\begin{cases}
			\pibar_{{ \chi_{\{f>t\}}}}({\nu}_f)\in\tanbvX{\{f>t\}}&\text{ for $\LL^1$-a.e.\ $t>0$},\\
			\pibar_{{ \chi_{\{f<t\}}}}({\nu}_f)\in\tanbvX{\{f<t\}}&\text{ for $\LL^1$-a.e.\ $t<0$},
		\end{cases}
	\end{equation}
where we considered $\nu_f$ as an element of $\tanXcap$.
Then $\nu_{\{f>t\}}=\nu_f^t$ for $\LL^1$-a.e.\ $t>0$ and $\nu_{\{f<t\}}=-\nu_f^t$ for $\LL^1$-a.e.\ $t<0$.
\end{lem}
\begin{proof}
Notice first that \eqref{coareaeqdiff} implies that $\nu_f^t$ is well defined for $\LL^1$-a.e.\ $t\in\RR$ and that $|{\nu_f^t}|=1$ $\abs{\DIFF\chi_{\{f>t\}}}$-a.e.\ for $\LL^1$-a.e.\ $t>0$ and $|{\nu_f^t}|=1$ $\abs{\DIFF\chi_{\{f<t\}}}$-a.e.\ for $\LL^1$-a.e.\ $t<0$.
	By Cavalieri's integration formula and \eqref{difffloccaompeq2},
\begin{equation}\label{cavvett}
	\int_\XX v\,\cdot\,\nu_f\dd{\abs{\DIFF f}}=\int_0^\infty \int_\XX v\,\cdot\,\nu_{\{f >t\}}\dd{\abs{\DIFF \chi_{\{f>t\}}}}\dd{t}-\int_{-\infty }^0\int_\XX v\,\cdot\,\nu_{\{f <t\}}\dd{\abs{\DIFF \chi_{\{f<t\}}}}\dd{t}
\end{equation}
for every $v\in\TestV(\XX)$.

By Lemma \ref{densitytestbar}, we have a sequence $\{v_k\}_k\subseteq \TestV(\XX)$ such that $v_k\rightarrow\nu_{f}$ in $\tanbvX{f}$, so that, using \eqref{coareaeq},
\begin{equation}\notag
	\begin{split}
		&\bigg|\int_0^\infty \int_\XX \nu_{\{f>t\}}\,\cdot\,\nu_f^t\dd{\abs{\DIFF\chi_{\{f>t\}}}}\dd{t}-\int_{-\infty}^0 \int_\XX \nu_{\{f<t\}}\,\cdot\,\nu_f^t\dd{\abs{\DIFF\chi_{\{f<t\}}}}\dd{t}\\&\quad-\int_0^\infty \int_\XX \nu_{\{f>t\}}\,\cdot\,v_k\dd{\abs{\DIFF\chi_{\{f>t\}}}}\dd{t}+\int_{-\infty}^0 \int_\XX \nu_{\{f<t\}}\,\cdot\,v_k\dd{\abs{\DIFF\chi_{\{f<t\}}}}\dd{t}\bigg| \\&\le\int_{-\infty}^\infty \int_\XX |{\nu_f^t -v_k}|\dd{\abs{\DIFF\chi_{\{f>t\}}}}\dd{t}=\int_\XX \abs{\nu_f-v_k}\dd{\abs{\DIFF f}}\rightarrow 0.
	\end{split}
\end{equation}
We can compute, using \eqref{coareaeqdiff} and \eqref{cavvett},
\begin{equation}\notag
	\begin{split}
		&\int_{0}^\infty \int_\XX |{\nu_{\{f>t\}}-\nu_f^t}|^2\dd{\abs{\DIFF\chi_{\{f>t\}}}}\dd{t}+\int_{-\infty}^0 \int_\XX |{\nu_{\{f<t\}}+\nu_f^t}|^2\dd{\abs{\DIFF\chi_{\{f<t\}}}}\dd{t}\\&\quad=2\abs{\DIFF f}(\XX)-2\int_0^\infty \int_\XX \nu_{\{f>t\}}\,\cdot\,\nu_f^t\dd{\abs{\DIFF\chi_{\{f>t\}}}}+2\int_{-\infty}^0 \int_\XX \nu_{\{f<t\}}\,\cdot\,\nu_f^t\dd{\abs{\DIFF\chi_{\{f<t\}}}}\\&\quad=2\abs{\DIFF f}(\XX)-2\lim_k\bigg(\int_0^\infty \int_\XX \nu_{\{f>t\}}\,\cdot\,v_k\dd{\abs{\DIFF\chi_{\{f>t\}}}} -\int_{-\infty}^0 \int_\XX \nu_{\{f<t\}}\,\cdot\,v_k\dd{\abs{\DIFF\chi_{\{f<t\}}}}\bigg) \\&\quad=2\abs{\DIFF f}(\XX)-2\lim_k\int_\XX \nu_f\,\cdot\,v_k\dd{\abs{\DIFF f}}=0.
	\end{split}
\end{equation}
which yields the conclusion.
\end{proof}

\bigskip

\begin{defn}
	Let $(\XX,\dist,\mass)$ be a metric measure space and $F=(F_1,\dots,F_n)\in\BVv^n$. We define the Borel set
	\begin{equation}\notag
		S_F\defeq\bigcup_{i=1}^n S_{F_i}.
	\end{equation}
\end{defn}

The following result is taken from \cite[Theorem 5.3]{ambmirpal04} and is needed in the proof of Proposition \ref{normalvectval}. We remark that in Proposition \ref{veryweakchain} we will prove that under the $\RCD(K,N)$  assumption, a more precise version holds.
\begin{thm}
Let $(\XX,\dist,\mass)$ be a $\PI$ space and $f\in\BVv$. Then
$$\abs{\DIFF f}=\abs{\DIFF f}\mres (\XX\setminus S_f)+\theta_f\HH^h\mres S_f$$
for some Borel map $\theta_f:S_f\rightarrow \RR$ that is $\HH^h$-a.e.\ strictly positive.
\end{thm}
As a consequence of the decomposition above, recalling also  \eqref{diffllxxxeq}, if $F\in\BVv^n$ it holds $$\abs{\DIFF F}\mres S_{F_i}\ll\HH^h\mres S_{F_i}\ll\abs{\DIFF F_i}\mres S_{F_i},$$
which implies that the statement in \eqref{normalbig} below is well posed.
\begin{prop}\label{normalvectval}
	Let $(\XX,\dist,\mass)$ be an $\RCD(K,N)$ space and let $F\in\BVv^n$. Then there exists a pair of $\abs{\DIFF F}$-measurable functions $F^l,\ F^r: \XX\rightarrow\RR^n$ and ${\nu}^S_F\in\tanbvX{F}$ such that
	\begin{enumerate}[label=\roman*)]
		\item for $\abs{\DIFF F}$-a.e.\ $x\in S_f$, $|{{\nu}^S_F}|(x)=1$ and $F^r(x)\ne F^l(x)$,
		\item for $\abs{\DIFF F}$-a.e.\ $x\in S_f$, for every $i=1,\dots,n$,
		\begin{equation}\label{normalbig}
			(F^r_i(x)-F^l_i(x)){\nu}^S_F(x)=(F_i^\vee(x)-F_i^\wedge(x)) \nu_{F_i}(x),
		\end{equation}
		where we considered ${\nu}^S_F$ and ${\nu}_{F_i}$ as elements of $\tanXcap$,
		\item for $\abs{\DIFF F}$-a.e.\ $x\in S_F$ there exists a set $E\subseteq\XX$ of finite perimeter such that $x\in\partial^* E$ and 
		\begin{equation}\label{asintoticaint}
			\lim_{r\searrow 0}\dashint_{B_r(x)\cap E} \abs{F- F^r(x)}\dd{\mass}=\lim_{r\searrow 0}\dashint_{B_r(x)\cap(\XX\setminus E)} |F- F^l(x)|\dd{\mass}=0.
		\end{equation}
	\item for $\abs{\DIFF F}$-a.e.\ $x\in\XX\setminus S_F$, $F^l(x)=F^r(x)$ and
	$$
	\lim_{r\searrow 0}\dashint_{B_r(x)}\abs{F-F^r(x)}\dd{\mass}=0.
	$$
	\end{enumerate}
	Finally, if $\tilde{F}^l, \tilde{F}^r, \tilde{\nu}_F^S$ is another triplet as above, then for $\abs{\DIFF F}\text{-a.e.\ }x\in S_f$ either it holds $$(\tilde{F}^l(x), \tilde{F}^r(x))=({F}^l(x), {F}^r(x))\quad\text{and}\quad\tilde{\nu}_F^S(x)=\nu_F^S(x)$$ or it holds $$(\tilde{F}^l(x), \tilde{F}^r(x))=({F}^r(x), {F}^l(x))\quad\text{and}\quad\tilde{\nu}_F^S(x)=-\nu_F^S(x),$$
	and, for $\abs{\DIFF F}\text{-a.e.\ }x\in\XX\setminus S_f$,
	$$\tilde{F}^l(x)= \tilde{F}^r(x)={F}^l(x)= {F}^r(x).$$
\end{prop}
\begin{proof}
	Using Lemma \ref{normsublevel} and standard considerations, we can find a countable dense subset of $\RR$, $\{t_j\}_{j\in\NN}$, $0\notin \{t_j\}_j$ such that for every $i=1,\dots,n$, we have
	\begin{enumerate}[label=\roman*)]
		\item for \begin{equation}\notag
			E_{i,j}=\begin{cases}
				\{F_i>t_j\}\quad&\text{if } t_j> 0, \\
				\{F_i<t_j\}  &\text{if } t_j<0,
			\end{cases}
		\end{equation}
		$E_{i,j}$ is a set of finite perimeter and finite mass,
		\item $S_{F_i}= \bigcup_{j\in\NN} \partial^* E_{i,j}$,
		\item for every $i=1,\dots,n$,
		\begin{equation}\notag
			\nu_{E_{i,j}}=
			\begin{cases}
				\pibar_{{\chi_{E_{i,j}}}}({\nu}_{F_i})\quad&\text{if }t_j>0, \\
				-\pibar_{{\chi_{E_{i,j}}}}({\nu}_{F_i})\quad&\text{if }t_j<0.
			\end{cases} 
		\end{equation}
	\end{enumerate}  
	We recall that thanks to \eqref{finitenessprec}, $F_i^\vee,\ F_i^\wedge,\ \overline{F_i}\in\RR$ for $\HH^h$-a.e.\ $x\in\XX$ and for every $i=1,\dots,n$.
	We set $F_i^l=F_i^r\defeq\overline{F_i}$ if $x\notin S_{F_i}$. We define $\nu_F^S$ as an element of $\tanXcap$, then it is enough to take the projection $\pibar_{F}$ to obtain a vector field of $\tanbvX{F}$ as in the statement.
	We define $F_i^l$, $F_i^r$ and ${\nu}_F^S$ iteratively. More precisely, we define $(F_1^r,F_1^l)=(F_1^\vee,F_1^\wedge)$ on $S_{F_1}$ and $\nu^S_F={\nu}_{F_1}\chi_{S_{F_1}}$. At step $k$, let $G_k\defeq \bigcup_{i=1}^{k-1} S_{F_i}$. We define $(F_k^r,F_k^l)=(F_k^\vee,F_k^\wedge)$ on $S_{F_k}\setminus G_k$ and we add to ${\nu}^S_F$ the $\capa$ vector field ${\nu}_{F_k}\chi_{S_{F_k}\setminus G_k}$. Now, Lemma \ref{oknormal} and the construction above imply that at $\HH^h$-a.e.\ $x\in S_{F_k}\cap G_k$ it holds ${\nu}_{F_k}=\pm\nu_F^S$, whence $(F_k^r,F_k^l)$ is uniquely defined on $S_{F_k}\cap G_k$ by the request $ii)$.
	
	We prove now \eqref{asintoticaint}.  First, arguing as in the proof of \cite[Theorem 3.5]{kinkorshatuo} and taking into account \eqref{finitenessprec}, we see that we can assume with no loss of generality that $F_i\in\Lpi$ for every $i=1,\dots,n$.
	Up to discarding an $\HH^h$-negligible set, we may restrict ourselves to the set of points of $S_F$ at which the density of every set $E_{i,j}$ is in $\{0,1/2,1\}$, by \eqref{unmezzoqoeq}.
	Notice that if $n=1$ and $x\in\partial^*{E_{i,j}}$, then \eqref{asintoticaint} holds with either $E_{i,j}$ or $\XX\setminus E_{i,j}$ in place of $E$. This follows from a standard argument as the one used at the end of Step 2 of the proof of Lemma \ref{convergenceprec}. But then the same conclusion holds also if $n\ge 1$, up to $\HH^h$-negligible subsets, by Lemma \ref{oknormal} (thanks to our choice of $F^l$ and $F^r$). Also, \cite[Theorem 3.5]{kinkorshatuo} proves the last item.
	
	We prove now uniqueness, in the sense explained at the end of the statement. In $\XX\setminus S_f$, it is clear, so let us focus on $S_f$. We just have to prove that at $\abs{\DIFF F}$-a.e.\ $x\in S_f$, $(\tilde{F}^l(x),\tilde{F}^r(x))$ coincides, up to the order, with $({F}^l(x),{F}^r(x))$, then we can use \eqref{normalbig} to conclude. We can assume that at $x$ there exist two sets of finite perimeter $E,\tilde{E}$ with $x\in\partial^*E\cap\partial^*\tilde{E}$ and such that \eqref{asintoticaint} holds and also the variant of \eqref{asintoticaint} for $\tilde{F}^l(x),\tilde{F}^r(x),\tilde{E}$ holds. Now, notice that it holds that 
	$$
	0<\limsup_{r\searrow 0} \frac{\mass(E\cap B_r(x))}{\mass(B_r(x))}\le \limsup_{r\searrow 0} \frac{\mass(E\cap\tilde{E}\cap B_r(x))}{\mass(B_r(x))}+\limsup_{r\searrow 0} \frac{\mass(E\setminus\tilde{E}\cap B_r(x))}{\mass(B_r(x))}.
	$$
	Therefore, either $0<\limsup_{r\searrow 0} \frac{\mass(E\cap\tilde{E}\cap B_r(x))}{\mass(B_r(x))}$ or $0<\limsup_{r\searrow 0} \frac{\mass(E\setminus\tilde{E}\cap B_r(x))}{\mass(B_r(x))}$. In the first case, we infer that $\tilde{F}^r(x)={F}^r(x)$, in the second case that $\tilde{F}^l(x)={F}^r(x)$. We can deal similarly with $F^l(x)$.
\end{proof}
\begin{rem}
	A careful inspection of the proof of Proposition \ref{normalvectval} shows that in item $iii)$ we can replace the integral $\dashint_{B_r(x)\cap E} \abs{F- F^r(x)}\dd{\mass}$ with $\dashint_{B_r(x)\cap E} \abs{F- F^r(x)}^{{Q}/(Q-1)}\dd{\mass}$ for any $Q=Q(R)$ given as in \eqref{poincare2} and similarly for the integral involving $F^l$. A similar consideration holds for item $iv)$.\fr
\end{rem}
\begin{defn}
	Let $(\XX,\dist,\mass)$ be an $\RCD(K,N)$ space and $F\in\BVv^n$. Define the functions $F^l,F^r:\XX\rightarrow\RR^n$ and the vector field ${\nu}^S_F$ as a suitable triplet given by Proposition \ref{normalvectval} above. Define $\bar{F}:\XX\rightarrow\RR^n$ as $$\bar{F}\defeq\frac{F^l+F^r}{2}.$$
\end{defn}
There may be more than one possible choice for the triplet $(F^l,F^r,{\nu}_F^S)$, however, the quantity $(F^l-F^r)\nu_F^S$ is well defined $\abs{\DIFF F}$-a.e.\ on $S_F$. As usual, we often consider $\nu_F^S$ as an element of $\tanXcap$ (defined $\abs{\DIFF F}\mres S_F$-a.e.).

\subsection{Calculus rules}\label{calcrules}
In this subsection, we prove that the usual calculus rules (Leibniz rule and chain rule) are satisfied by functions of bounded variation on $\RCD(K,N)$ spaces. Namely, we study the behaviour of $\DIFF(\phi\circ F)$ and $\DIFF (f g)$ for suitable functions. Recall the definition of distributional differential given in Subsection \ref{sectBVRCD} and the algebraic rules in Definition \ref{formalcalc}.

We start by proving the chain rule for scalar functions of bounded variation. We first show the chain rule for the total variation, which relies on the coarea formula and a change of variables, following \cite{Ambanewproof}. The full chain rule is then obtained via an integration by Cavalieri's formula.
\begin{prop}[Chain rule]\label{volpertprop}
	Let $(\XX,\dist,\mass)$ be an $\RCD(K,N)$ space, $f\in\BVv$ and $\phi\in \LIP(\RR)$ such that $\phi(0)=0$. Then $\phi\circ f\in\BVv$ and%
\begin{equation}\label{volpert}
\DIFF(\phi\circ f)=\left(\int_0^1 \phi'(t f^\vee+(1-t) f^\wedge)\dd{t}\right)\DIFF f.
\end{equation}
\end{prop}
We comment briefly on the well-posedness of \eqref{volpert}. Recalling \eqref{finitenessprec}, we see that it suffices then to check that \begin{equation}\notag
	\abs{\DIFF f} (A)=0,
\end{equation}
where $$A\defeq \left\{x\in\XX\setminus S_f: \phi\text{ is not differentiable at } \bar{f}(x)\right\}.$$

We can then use  \eqref{coareaeqdiff}, the relations in \eqref{relations}, \eqref{reprformula3} and finally Rademacher's Theorem to compute
\begin{equation*}
	\begin{split}
		\abs{\DIFF f} (A)&=\int_\RR \abs{\DIFF \chi_{\{f>t\}}}(A) \dd{t} =\frac{\omega_{n-1}}{\omega_n}\int_\RR \HH^h(A\cap \partial^*\{f>t\}) \dd{t}\\&= \frac{\omega_{n-1}}{\omega_n}\int_\RR \HH^h(\{x\in\XX\setminus S_f:\phi\text{ is not differentiable at $t$ and $\bar{f}(x)=t$}\}) \dd{t}=0.
	\end{split}
\end{equation*}
\begin{proof}
	We just have to show 	
	 \begin{equation}\notag
			\int_\XX \phi\circ f  \dive\, v\dd{\mass}=-\int_\XX v\,\cdot\,\nu_f \phi'(\bar{f})\dd{\abs{\DIFF f}}\mres (\XX\setminus S_f)-\int_\XX v\,\cdot\,\nu_f\frac{\phi(f^\vee)-\phi (f^\wedge)}{f^\vee-f^\wedge}\dd{\abs{\DIFF f}}\mres S_f
	\end{equation}
for every $v\in\TestV(\XX)$. 
	Using linearity, we can assume that $\phi$ is also bilipschitz and strictly increasing with no loss of generality.
	For the first part of the proof we follow the arguments contained in \cite{Ambanewproof,ambmirpal04}. 
	 Clearly, $\phi\circ f\in\BVv$. 
	Notice also that $\abs{\DIFF (\phi\circ f)}\ll\abs{\DIFF  f}\ll\abs{\DIFF (\phi\circ f)}$ and $S_{\phi\circ f}=S_f$. 
	
	We show now 
	\begin{equation}\label{volpabs}
		\abs{\DIFF(\phi\circ f)}=\phi'(\bar{f}){\abs{\DIFF f}}\mres (\XX\setminus S_f)+\frac{\phi(f^\vee)-\phi (f^\wedge)}{f^\vee-f^\wedge}{\abs{\DIFF f}}\mres S_f.
	\end{equation}
	The following relations can be easily proved with standard measure theoretic arguments (see e.g.\ \cite[Proposition 5.2]{ambmirpal04}):
	\begin{equation}\label{relations}
		\begin{aligned}
			&\text{if }x\in S_f \text{ and }t\in (f^\wedge(x),f^\vee(x))\quad&&\text{then } x\in\partial^*\{f>t\},\\
			&\text{if } x\in\partial^*\{f>t\} &&\text{then }t\in [f^\wedge(x),f^\vee(x)],\\
			&\qquad\text{in particular }\\
			&\text{if }x\notin S_f\text{ and }x\in\partial^*\{f>t\} 		&&\text{then } \bar{f}(x)=t.\\
		\end{aligned}
	\end{equation}

	Take $B\subseteq \XX\setminus S_f$ Borel. Using the coarea formula \eqref{coareaeqdiff} together with a change of variables, we obtain
	\begin{equation}\label{volpabs1}
		\begin{split}
			\abs{\DIFF (\phi\circ f)}(B)&=\int_\RR \per(\{\phi\circ f>t\},B)\dd{t}=\int_\RR \int_B \phi'(t)\dd{ \per(\{ f>t\},\,\cdot\,)}\dd{t}\\&=\int_\RR \int_B \phi'(\bar{f}(x))\dd{ \per(\{ f>t\},\,\cdot\,)}\dd{t}=\int_B \phi'(\bar{f}(x))\dd{\abs{\DIFF f}}
		\end{split}
	\end{equation}
	where in the next to last equality we used \eqref{relations}, thanks to \eqref{reprformula3}, and the last equality is a consequence of the coarea formula \eqref{coareaeq} as in \cite[Proposition 5.4]{ambmirpal04}.
	Take now $B\subseteq S_f$ Borel. Using \eqref{coareaeqdiff} together with a change of variables we obtain
	\begin{equation}\notag
		\begin{split}
			\abs{\DIFF (\phi\circ f)}(B)&=\int_\RR \per(\{\phi\circ f>t\},B)\dd{t}=\int_\RR \int_B \phi'(t)\dd{ \per(\{ f>t\},\,\cdot\,)}\dd{t}
			\\&=\frac{\omega_{n-1}}{\omega_n}\int_\RR\int_B \phi'(t)\chi_{\partial^*\{f>t\}}(x)\dd{\HH^h\mres S_f}(x)\dd{t}
			\\&=\frac{\omega_{n-1}}{\omega_n}\int_B\int_\RR \phi'(t)\chi_{\partial^*\{f>t\}}(x)\dd{t}\dd{\HH^h \mres S_f}(x)
		\end{split}
	\end{equation}
	where in the next to last equality we used \eqref{reprformula3}. In the application of Fubini's theorem, the measurability of the map $(t,x)\mapsto\chi_{\partial^*\{f>t\}}(x)$ can be proved using standard arguments: just notice that for any $r$ the maps $$(x,t)\mapsto \frac{\mass (B_r(x)\cap\{f>t\})}{\mass( B_r(x))}\quad\text{and}\quad(x,t)\mapsto\frac{\mass (B_r(x)\setminus\{f>t\})}{\mass( B_r(x))}$$
	are continuous everywhere up to a set of null $( \HH^h\mres S_f)\otimes\LL^1$ measure.
	 Using also \eqref{relations},
	\begin{equation}\notag
	\frac{\omega_{n-1}}{\omega_n}	\int_B\int_\RR \phi'(t)\chi_{\partial^*\{f>t\}}(x)\dd{t}\dd{\HH^h}(x)=
		\frac{\omega_{n-1}}{\omega_n}\int_B\int_{f^\wedge(x)}^{f^\vee(x)} \phi'(t)\dd{t}\dd{\HH^h}(x).
	\end{equation}
	Therefore
	\begin{equation}\label{volpabs2}
		\abs{\DIFF (\phi\circ f)}(B)=\frac{\omega_{n-1}}{\omega_n}\int_B (\phi(f^\vee(x))-\phi(f^\wedge(x)))\dd{\HH^h}(x).
	\end{equation}
	Taking into account \eqref{volpabs1} and \eqref{volpabs2}, also with $\phi(s)=s$, we conclude the proof of \eqref{volpabs}.
	
	Using Cavalieri's integration formula together with a change of variables, taking into account Lemma \ref{normsublevel},
	\begin{equation}\notag
		\begin{split}
			&-\int_\XX\phi\circ f\dive\, v \dd{\mass}=-\int_0^\infty \int_{\{\phi\circ f>t\}} \dive\, v\dd{\mass}\dd{t}+\int_{-\infty}^0 \int_{\{\phi\circ f<t\}} \dive\, v\dd{\mass}\dd{t}\\&\quad
			=\int_0^\infty \phi'(t)\int_\XX v\,\cdot\, \nu_{\{ f>t\}}\dd{\abs{\DIFF\chi_{\{ f>t\}}}}\dd{t}-\int_{-\infty}^0\phi'(t) \int_\XX v\,\cdot\, \nu_{\{ f<t\}}\dd{\abs{\DIFF\chi_{\{ f<t\}}}}\dd{t}
			\\&\quad=\int_0^\infty \phi'(t)\int_\XX v\,\cdot\, {\nu}_f\dd{\abs{\DIFF\chi_{\{ f>t\}}}}\dd{t}+\int_{-\infty}^0\phi'(t) \int_\XX v\,\cdot\, {\nu}_f\dd{\abs{\DIFF\chi_{\{ f>t\}}}}\dd{t}\\&\quad=\int_\RR \int_\XX v\,\cdot\, {\nu}_f \phi'(t)\dd{\abs{\DIFF\chi_{\{f>t\}}}}\dd{t}=\int_\RR \int_\XX v\,\cdot\, {\nu}_f \dd{\abs{\DIFF\chi_{\{\phi\circ f>t\}}}}\dd{t}\\&\quad= \int_\XX v\,\cdot\,\nu_f \dd{\abs{\DIFF( \phi\circ f)}}
		\end{split}
	\end{equation}
	for every $v\in\TestV(\XX)$, where the last equality above is due to \eqref{coareaeq}.	
%
\end{proof}

The following proposition follows as intermediate step in the computation of $\abs{\DIFF(\phi\circ f)}$ in \eqref{volpabs2} (taking $\phi(s)=s$).
\begin{prop}\label{veryweakchain}
	Let $(\XX,\dist,\mass)$ be an $\RCD(K,N)$ space of essential dimension $n$ and let $f\in\BVv$. Then 
	\begin{equation}\label{dfmressf}
		\abs{\DIFF f}\mres S_f=(f^\vee-f^\wedge)\frac{\omega_{n-1}}{\omega_n}\HH^h\mres S_f.
	\end{equation}
\end{prop}

In the following proposition, we restrict ourselves to the case $f,g$ bounded functions of bounded variation although the boundedness hypothesis can be slightly weakened using approximation arguments as done in the proof of Proposition \ref{volprop} below. 
\begin{prop}[Leibniz rule]\label{leibnizprop}
	Let $(\XX,\dist,\mass)$ be an $\RCD(K,N)$ space and $f,g\in\BVv\cap\Lpi$. Then $f g \in\BVv$ and
	\begin{equation}\label{leibniz}
	\DIFF (f g) =\bar{f}\DIFF g+\bar{g}\DIFF f.
\end{equation}
	In particular, $\abs{\DIFF (f g)}\le \abs{\bar{f}}\abs{\DIFF g}+ \abs{\bar{g}}\abs{\DIFF f}$.
\end{prop}
\begin{proof}
	Using the chain rule of Proposition \ref{volpertprop} with $\phi\in\LIP(\RR)$ that coincides with $t\mapsto t^2$ on a sufficiently large neighbourhood of $0$, we see that 
	\begin{align}
		\DIFF (f+g)^2&=2(\overline{f+g})\DIFF (f+g)=2(\overline{f}+\overline{g})\DIFF (f+g)\label{leib1},\\
		\DIFF f^2&=2 \overline{f}\DIFF f\label{leib2},\\
		\DIFF g^2&=2 \overline{g}\DIFF g\label{leib3}.
	\end{align}
Here we used that $\overline{f+g}=\overline{f}+\overline{g} \ \HH^h$-a.e.\ which follows e.g.\ from Lemma \ref{convergenceprec}.
	Using the linearity of the map $f\mapsto \DIFF f$, subtracting \eqref{leib2} and \eqref{leib3} from \eqref{leib1}, we obtain \eqref{leibniz}.
\end{proof}
\begin{prop}
Let $(\XX,\dist,\mass)$ be an $\RCD(K,N)$ space and $f,g\in\BVv\cap\Lpi$.
Then \begin{equation}\notag	\lim_{t\searrow 0} \int_\XX g v\cdot\nabla\heat_t f\dd{\mass}=\int_\XX \bar{g} v \,\cdot\,\nu_f \dd{\abs{\DIFF f}}
\end{equation}
for every $v\in \Cqcvfinf\cap D(\dive)$.
\end{prop}
\begin{proof}
We can write, thanks to the calculus rules,
\begin{equation}\notag
\int_\XX \heat_t f\heat_s g\dive\, v \dd{\mass}=-\int_\XX\heat_s g \nabla\heat_t f\,\cdot\, v \dd{\mass}-\int_\XX\heat_t f\nabla \heat_s g \,\cdot\, v\dd{\mass}.
\end{equation}
We let now first $s\searrow 0$ then $t\searrow 0$, use Lemma \ref{convergenceprec} and compare the outcome with the result given by \eqref{leibniz}.
\end{proof}

\begin{lem}\label{jumpvect}
	Let $(\XX,\dist,\mass)$ be an $\RCD(K,N)$ space of essential dimension $\mu$ and let $F\in\BVv^n$. Then 
	\begin{equation}\notag
		\DIFF F\mres S_F=(F^r-F^l)\nu_F^S\frac{\omega_{\mu-1}}{\omega_\mu}\HH^h\mres S_F.
	\end{equation}
\end{lem}
\begin{proof}
By Proposition \ref{normalvectval}, \eqref{dfmressf} and \eqref{coordnorm} it is enough to show that $$\abs{\DIFF F}\mres S_F=|{F^r-F^l}|\frac{\omega_{\mu-1}}{\omega_\mu}\HH^h\mres S_F.$$

First notice that as $\abs{\DIFF F}\le\sum_{i=1}^n \abs{\DIFF F_i}$ and by \eqref{dfmressf} we know that $\abs{\DIFF F}\mres S_F=\rho\frac{\omega_{\mu-1}}{\omega_\mu}\HH^h\mres S_F$ for some $\rho\in L^1(\HH^h\mres S_F)$. Also, by \eqref{coordnorm}, $$\sum_{i=1}^n \left(\dv{\abs{\DIFF F_i}}{\abs{\DIFF F}}\right)^2=1\quad\abs{\DIFF F}\text{-a.e.}$$
which reads, recalling \eqref{dfmressf}
$$
\sum_{i=1}^n\left( \frac{(F_i^\vee-F_i^\wedge)}{\rho}\right)^2=1\quad\abs{\DIFF F}\text{-a.e.}$$
ant then yields the claim.
\end{proof}

We state now the main result of this section, namely the Vol'pert chain rule formula, proved in the smooth setting in \cite{Vol_pert_1967} (see also \cite{vol1985analysis}). Here we adopt a completely different proof, to avoid the original blow-up argument. We state the formula \eqref{volpert1} in the form of the so called Vol'pert averaged superposition. 
\begin{thm}[Chain rule for vector valued functions]\label{volprop}
	Let $(\XX,\dist,\mass)$ be an $\RCD(K,N)$ space and $F\in\BVv^n$. Let $\phi\in C^1(\RR^n;\RR^m)\cap\LIP(\RR^n;\RR^m)$ for some $m\in\NN$, $m\ge 1$ such that $\phi(0)=0$.
	Then $\phi\circ F\in\BVv^m$ and 
	\begin{equation}\label{volpert1}
	\DIFF(\phi\circ F)=\left(\int_0^1 \nabla\phi(t F^r+(1-t) F^l)\dd{t}\right)\DIFF F.
	\end{equation}

\end{thm}

\begin{proof}
	First of all, 	recall (e.g.\ \cite[Theorem 5.3]{ambmirpal04}) that for every $i=1,\dots,n$, $\abs{\DIFF F_i}\mres(\XX\setminus S_{F_i})$ vanishes on sets that are $\sigma$-finite with respect to $\HH^h$, in particular on $S_F$. 
	By Remark \ref{normcomp}, we see that we can assume $m=1$ with no loss of generality.
	
Recalling the very definition of $\DIFF F$ and \eqref{coordnorm}, it suffices to prove
\begin{equation}\label{toprovevvchain}
		\begin{split}
	&\int_\XX \phi\circ F  \dive\, v\dd{\mass}=-\sum_{i=1}^n\int_\XX v\,\cdot\,\nu_{F_i} \left(\int_0^1 \partial_i\phi(t F^r+(1-t) F^l)\dd{t}\right)\dd{\abs{\DIFF F_i}}
\end{split}
\end{equation}
for every $v\in\TestV(\XX)$. Let $\mu$ denote the essential dimension of the space.

	Assume for the moment $F\in(\BVv\cap\Lpi)^n$, say $\abs{F_i}\le L\ \mass$-a.e.\ for every $i=1,\dots,n$, for some $L>0$. 
We start by showing that
\begin{equation}\label{firstclaim}
		\begin{split}
			\int_\XX \phi\circ F  \dive\, v\dd{\mass}&=-\sum_{i=1}^n\int_\XX v\,\cdot\,\nu_{F_i} \partial_i\phi(\bar{F})\dd{\abs{\DIFF F_i}}\mres (\XX\setminus S_F)\\&\qquad-\int_\XX v\,\cdot\,\nu_{F}^S  (\phi(F^r)-\phi(F^l))\frac{\omega_{\mu-1}}{\omega_\mu}\dd{\HH^h}\mres S_F
		\end{split}
	\end{equation}
for every $v\in\TestV(\XX)$. Then \eqref{toprovevvchain} follows in the case $F\in(\BVv\cap\Lpi)^n$, using Lemma \ref{jumpvect} and the construction of $\nu_F^S$ (Proposition \ref{normalvectval}).

Notice that if \eqref{firstclaim} holds for a sequence $\{\phi_k\}_k\subseteq C^1(\RR^n)\cap\LIP(\RR^n)$ with $\phi_k(0)=0$ and $(\phi_k,\nabla\phi_k)\rightarrow( \phi,\nabla\phi)$  uniformly on $[-L,L]^n$, then  \eqref{firstclaim} holds also for $\phi$.

Let $\epsilon>0$.
By a mollification and cut-off argument, we find $\tilde{\phi}\in C^\infty(\RR^n)$ such that $\supp \phi\subseteq [-2L,2L]^n$, $\tilde{\phi}(0)=0$ and
\begin{equation}\notag
\sup_{x\in [-L,L]^n} |{\phi(x)-\tilde{\phi}(x)}|<\epsilon\quad\text{and}\quad \sup_{x\in [-L,L]^n} |{\nabla\phi(x)-\nabla\tilde{\phi}(x)}|<\epsilon.
\end{equation} Now, by the Stone--Weierstrass Theorem, we find a polynomial $g:\RR^n\rightarrow\RR$ such that
\begin{equation}\notag
\sup_{x\in [-2 L, 2 L]^n} |{\partial_1\cdots\partial_n\tilde{\phi}(x)}-g(x)|<\epsilon.
\end{equation}
Set now $$\hat{\phi}((x_1,\dots,x_n))\defeq\int_{-2 L}^{x_1}\dd{s_1}\cdots\int_{-2 L}^{x_n}\dd{s_n} g((s_1,\dots, s_n)),$$
it is not hard to verify that
\begin{equation}\notag
	\sup_{x\in [-L,L]^n} |{\tilde{\phi}(x)-\hat{\phi}(x)}|<C\epsilon\quad\text{and}\quad \sup_{x\in [-L,L]^n} |{\nabla\tilde{\phi}(x)-\nabla\hat{\phi}(x)}|<C \epsilon,
\end{equation}
where $C$ depends only on $L$ and $n$. Eventually adding to $\hat{\phi}$ a constant, we can assume that $\hat{\phi}(0)=0$.

As $\hat{\phi}$ is a polynomial, we see that we can prove \eqref{firstclaim} for a polynomial  $\phi$ (to be more precise, with a function that coincides with a polynomial on a sufficiently large neighbourhood of $0$).
	Therefore, using also linearity, we see that we can assume with no loss of generality that $\phi$ is a monomial (to be more precise, $\phi$ coincides with a monomial on a sufficiently large neighbourhood of $0$). Also, up to changing $n$ and repeating some function $F_i$, we can assume that $$\phi(x_1,\dots,x_n)=x_1\cdots x_n.$$
	
	We proceed by induction on $n$. If $n=1$ the conclusion follows from item $ii)$ of Proposition \ref{normalvectval}, \eqref{dfmressf} and Theorem \ref{weakder2}.
	
We check now the inductive step. By \eqref{leibniz}
	\begin{equation}\label{tmpvp}
			-\int_\XX  F_1\cdots F_n  \dive\, v\dd{\mass}= \int_\XX  \overline{F_n} v\,\cdot\, \nu_{F_1\cdots F_{n-1}}\abs{\DIFF(F_1\cdots F_{n-1})}+\int_\XX\overline{F_1\cdots F_{n-1}} v\,\cdot\nu_{F_n}\abs{\DIFF F_n}.
	\end{equation}
	We set $G\defeq(F_1,\dots ,F_{n-1})$.
	Using Lemma \ref{convergenceprec} to approximate $F_n$ and the inductive hypothesis, we can rewrite the first addend on the right hand side of \eqref{tmpvp} as 
	\begin{equation}\notag
		\begin{split}
			&\sum_{i=1}^{n-1}\int_\XX \overline{F_n}  v\,\cdot\,\nu_{F_i} \overline{F_1}\cdots\overline {F_{i-1}}\,\overline {F_{i+1}}\cdots \overline{F_{n-1}}\dd{\abs{\DIFF F_i}}\mres (\XX\setminus S_G)\\&\quad+\frac{\omega_{\mu-1}}{\omega_\mu}\int_\XX \overline{F_n}v\,\cdot\, {\nu}_{G}^S  (F_1^r\cdots F_{n-1}^r-F_1^l\cdots F_{n-1}^l)\dd{\HH^h}\mres S_G.
		\end{split}
	\end{equation}

	Using also Proposition \ref{normalvectval}, we can rewrite the second addend on the right hand side of \eqref{tmpvp} as
	\begin{equation}\notag
		\begin{split}
			&\int_\XX\overline{F_1}\cdots \overline{F_{n-1}} v\,\cdot\nu_{F_n}\abs{\DIFF F_n}\mres(\XX\setminus S_{F_n})+\frac{\omega_{\mu-1}}{\omega_\mu}\int_\XX \overline{F_1 \cdots F_{n-1}} v\,\cdot\,\nu_{F_n} (F_n^\vee-F_n^\wedge)\dd{\HH^h}\mres S_{F_n}. 
		\end{split}
	\end{equation}
	Now, by the very construction of $\nu_F^S$ (see the proof of Proposition \ref{normalvectval}),
	\begin{equation}\notag
		\begin{split}
			&\int_\XX \overline{F_n}v\,\cdot\,{ \nu}_{G}^S  (F_1^r\cdots F_{n-1}^r-F_1^l\cdots F_{n-1}^l)\dd{\HH^h}\mres S_G\\&=\frac{1}{2}\int_\XX v\,\cdot\,{ \nu}_{F}^S (F_n^r+F_n^l)(F_1^r\cdots F_{n-1}^r-F_1^l\cdots F_{n-1}^l)\dd{\HH^h}\mres S_F
		\end{split}
	\end{equation}
	and also
	\begin{equation}\notag
		\begin{split}
			&\int_\XX \overline{F_1 \cdots F_{n-1}} v\,\cdot\,\nu_{F_n} (F_n^\vee-F_n^\wedge)\dd{\HH^h}\mres S_{F_n}\\&=
			\frac{1}{2} \int_\XX v\,\cdot\,{ \nu}_{F}^S (F_n^r-F_n^l)(F_1^r\cdots F_{n-1}^r+F_1^l\cdots F_{n-1}^l)\dd{\HH^h}\mres S_F,
		\end{split}
	\end{equation}
where we used $$(F_1\cdots F_{n-1})^\vee+(F_1\cdots F_{n-1})^\wedge=F_1^r\cdots F_{n-1}^r+F_1^l\cdots F_{n-1}^l\quad\HH^h\text{-a.e.\ on } S_F$$ as a consequence of Proposition \ref{normalvectval} applied to $(F_1,\dots, F_{n-1},F_1\cdots F_{n-1})$.
	Putting all this together, we conclude the proof of the inductive step.
	
	Now we get rid of the assumption $F\in\Lpi^n$ using an approximation argument. In this procedure, we consider the approximating sequence $\{F^m\}_m$ as in \eqref{approxsucc}. Now we can let $m\rightarrow\infty$ in  \eqref{toprovevvchain} for $F^m$, recalling Lemma \ref{normcut}, \eqref{finitenessprec} and \eqref{coareaeqdiff}. 
\end{proof}

\bibliographystyle{alpha}
\bibliography{Biblio11}
\end{document}